\documentclass[a4paper,12pt]{article}
\usepackage{latexsym,amsmath,amssymb,amsthm,enumerate,eucal,subfig,array,
  ctable}

\usepackage{graphicx}
\usepackage{xcolor}
\newtheorem{theorem}{Theorem}
\newtheorem{lemma}[theorem]{Lemma}
\newtheorem{proposition}[theorem]{Proposition}

\newtheorem{conjecture}[theorem]{Conjecture}

\newtheorem{observation}[theorem]{Observation}

\newcommand\Setx[1] {\left\{{#1}\right\}}

\newcommand\size[1] {\left|{#1}\right|}

\newcommand{\prob}[1]{\mathbf{P}(#1)} %
\newcommand{\cprob}[2]{\mathbf{P}(#1\ |\ #2)} %
\newcommand{\eps}{\varepsilon}

\newcommand{\pair}[3]{(#1,#2)^{#3}}

\newcommand{\bd}[1]{\partial(#1)}
\newcommand{\nbrx}[1]{\tilde N(#1)}
\newcommand{\nbrxc}[1]{\tilde N[#1]}
\newcommand{\vectrans}[1]{{\vec #1}\hspace{2pt}^T}

\newcommand{\mate}[2]{({#1}_{#2})'}

\newcommand{\fig}[1]{\includegraphics[page=#1]{fsub-figs.pdf}}%
\newcommand{\sfig}[2]{\subfloat[#2]{\fig{#1}}}%
\newcommand{\sfigdef}[1]{\newbox{\base}\sbox{\base}{\fig{#1}}}
\newcommand{\sfigtop}[2]{\newbox{\pic}\sbox{\pic}{\fig{#1}}%
  \subfloat[#2]{\vbox to\ht\base{\hbox to\wd\pic{\usebox\pic}}}}
\newcommand{\hf}{\hspace*{0pt}\hspace{\fill}\hspace*{0pt}}

\newtheorem{xcasehdr}{Case}
\newtheorem{xxcasehdr}{Subcase}[xcasehdr]

\newenvironment{xcase}[1]%
{\vspace{-1mm}\par\noindent
  \xcasehdr{#1}\upshape
  \vspace{2mm}\par\noindent}
{\hspace*{0mm}\hspace{\fill}$\blacktriangle$\par\vspace{2mm}}

\newlength{\zero}
\newenvironment{xxcase}[1]%
{
  \par\noindent
  \hangindent\parindent
  \hangafter=0
  \xxcasehdr{#1}\vspace{2mm}\par\upshape\noindent}
{\hspace*{0mm}\hspace{\fill}$\triangle$\par\vspace{4mm}}

\title{\textbf{The fractional chromatic number\\of triangle-free
    subcubic graphs}\footnote{This research was supported by project
    GA\v{C}R 201/09/0197 of the Czech Science Foundation.}} %
\author{David G. Ferguson$^1$\and Tom\'{a}\v{s} Kaiser$^2$\and Daniel
  Kr\'{a}l'$^3$}

\date{Revision R2 (April 27, 2013)}

\begin{document}
\maketitle

\footnotetext[1]{School of Business, University of Buckingham, Hunter
  Street, Buckingham, MK18 1EG, UK and Department of Mathematics,
  London School of Economics, Houghton Street, London, WC2A 2AE,
  UK. Email: \texttt{david.ferguson@buckingham.ac.uk} or
  \texttt{d.g.ferguson@lse.ac.uk}. Work on this paper was partly done
  during a visit to Department of Applied Mathematics, Charles
  University, Prague, Czech Republic.}%
\footnotetext[2]{Department of Mathematics, Institute for Theoretical
  Computer Science and NTIS--New Technologies for the Information
  Society (European Centre of Excellence), University of West Bohemia,
  Univerzitn\'{\i}~8, 306~14~Plze\v{n}, Czech Republic. Supported by
  project P202/12/G061 of the Czech Science Foundation. E-mail:
  \texttt{kaisert@kma.zcu.cz}.}%
\footnotetext[3]{Institute of Mathematics, DIMAP and Department of
  Computer Science, University of Warwick, Coventry CV4 7AL, United
  Kingdom. Previous affiliation: Institute of Computer Science (IUUK),
  Faculty of Mathematics and Physics, Malostransk\'{e}
  n\'{a}m\v{e}st\'{\i} 25, 118~00~Prague, Czech Republic. E-mail:
  \texttt{D.Kral@warwick.ac.uk}.}

\begin{abstract}
  Heckman and Thomas conjectured that the fractional chromatic number
  of any triangle-free subcubic graph is at most $14/5$. Improving on
  estimates of Hatami and Zhu and of Lu and Peng, we prove that the
  fractional chromatic number of any triangle-free subcubic graph is
  at most $32/11 \approx 2.909$.
\end{abstract}

\section{Introduction}
\label{sec:intro}

When considering the chromatic number of certain graphs, one may
notice colourings which are best possible (in that they use as few
colours as possible) but which are in some sense wasteful. For
instance, an odd cycle cannot be properly coloured with two colours
but can be coloured using three colours in such a way that the third
colour is used only once.

Indeed if $C_7$ has vertices $v_1,v_2,v_3,\dots,v_7$, then we can
colour $v_1,v_3,v_5$~red, $v_2,v_4,v_6$ blue and $v_7$ green. If,
however, our aim is instead to assign multiple colours to each vertex
such that adjacent vertices receive disjoint lists of colours, then we
could double-colour $C_7$ using five (rather than six) colours and
triple-colour it using seven (rather than nine) colours in such a way
that each colour is used exactly three times --- colour $v_i$ with
colours $3i, 3i+1, 3i+2$ (mod $7$). Asking for the minimum of the
ratio of colours required to the number of colours assigned to each
vertex gives us a generalisation of the chromatic number.

Alternatively, for a graph $G=(V,E)$ we can consider a function $w$
assigning to each independent set of vertices $I$ a real number
$w(I)\in [0,1]$. We call such a function a \textit{weighting}. The
weight $w[v]$ of a vertex $v\in V$ with respect to $w$ is then defined
to be the sum of $w(I)$ over all independent sets containing $v$. A
weighting $w$ is a \textit{fractional colouring} of $G$ if for each
$v\in V$ $w[v] \geq 1$. The size $|w|$ of a fractional colouring is
the sum of $w(I)$ over all independent sets $I$. The fractional
chromatic number $\chi_{f}(G)$ is then defined to be the infimum of
$|w|$ over all possible fractional colourings. We refer the reader to
\cite{bib:SU-fractional} for more information on fractional colourings
and the related theory.

By a folklore result, the above two definitions of the fractional
chromatic number are equivalent to each other and to a third,
probabilistic, definition. It is this third definition which we will
make most use of:

\begin{lemma}
\label{l:fraccol}
Let $G$ be a graph and $k$ a positive rational number. The following
are equivalent:
\begin{enumerate}[\quad(i)]
\item $\chi_f(G) \leq k$,
\item there exists an integer $N$ and a multi-set $\cal W$ of $kN$
  independent sets in $G$ such that each vertex is contained in
  exactly $N$ sets of $\cal W$,
\item there exists a probability distribution $\pi$ on the independent
  sets of $G$ such that for each vertex $v$, the probability that $v$
  is contained in a random independent set (with respect to $\pi$) is
  at least $1/k$.
\end{enumerate}
\end{lemma}

In this paper, we consider the problem of bounding the fractional
chromatic number of a graph that has maximum degree at most three (we
call such graphs \emph{subcubic}) and contains no triangle. Brooks'
theorem (see, e.g., \cite[Theorem 5.2.4]{bib:Die-graph}) asserts that
such graphs have chromatic number at most three, and, thus, also have
fractional chromatic number at most three. On the other hand,
Fajtlowicz~\cite{bib:Faj-size} observed that the independence number
of the generalised Petersen Graph $P(7,2)$ (Figure~\ref{fig:gp72})
equals 5, which implies that $\chi_f(P(7,2))=14/5=2.8$.

\begin{figure}
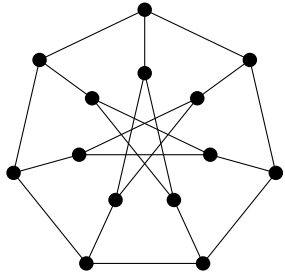

  \centering
  \fig{60}
  \caption{The generalised Petersen graph.}
  \label{fig:gp72}
\end{figure}

In 2001, Heckman and Thomas~\cite{bib:HT-new} made the following
conjecture:

\begin{conjecture}\label{conj:ht}
  The fractional chromatic number of any triangle-free subcubic graph
  $G$ is at most $2.8$.
\end{conjecture}
Conjecture~\ref{conj:ht} is based on the result of
Staton~\cite{bib:Sta-some} (see also \cite{bib:Jon-size,bib:HT-new})
that any triangle-free subcubic graph contains an independent set of
size at least $5n/14$, where $n$ is the number of vertices of $G$. As
shown by the graph $P(7,2)$, this result is optimal.

Hatami and Zhu~\cite{bib:HZ-fractional} proved that under the same
assumptions, $\chi_f(G)\leq 3-3/64 \approx 2.953$. More recently, Lu
and Peng \cite{bib:LP-fractional} were able to improve this bound to
$\chi_f(G)\leq 3- 3/43 \approx 2.930$. We offer a new probabilistic
proof which improves this bound as follows:

\begin{theorem}
  \label{t:main}
  The fractional chromatic number of any triangle-free subcubic graph
  is at most $32/11 \approx 2.909$.
\end{theorem}

We remark that while this paper was under review, Dvo\v{r}\'{a}k,
Sereni and Volec~\cite{bib:DSV-subcubic} succeeded in proving
Conjecture~\ref{conj:ht}. Their result was preceded by an improvement
of the bound in Theorem~\ref{t:main} to $43/15 \approx 2.867$ due to
Liu~\cite{bib:Liu-upper}.

In the rest of this section, we review the necessary terminology. The
\emph{length} of a path $P$, denoted by $\size P$, is the number of
its edges. We use the following notation for paths. If $P$ is a path
and $x,y\in V(P)$, then $xPy$ is the subpath of $P$ between $x$ and
$y$. The same notation is used when $P$ is a cycle with a specified
orientation, in which case $xPy$ is the subpath of $P$ between $x$ and
$y$ which follows $x$ with respect to the orientation. In both cases,
we write $d_P(x,y)$ for $\size{xPy}$.

We distinguish between edges in undirected graphs and arcs in directed
graphs. If $xy$ is an arc, then $x$ is its \emph{tail} and $y$ its
\emph{head}.

If $G$ is a graph and $X,Y\subseteq V(G)$, then $E(X,Y)$ is the set of
edges of $G$ with one endvertex in $X$ and the other one in $Y$. We
let $\bd X$ denote the set $E(X,V(G)-X)$. For a subgraph $H\subseteq
G$, we write $\bd H$ for $\bd{V(H)}$, and we extend the definition of
the symbol $E(X,Y)$ to subgraphs in an analogous way. The
\emph{neighbourhood} of a vertex $u$ of $G$ is the set $N(u)$ of its
neighbours. We define $N[u] = N(u)\cup\Setx u$ and call this set the
\emph{closed neighbourhood} of $u$.


\section{An algorithm}
\label{sec:algorithm}

Let $G$ be a simple cubic bridgeless graph. By a well-known theorem of
Petersen (see, e.g., \cite[Corollary~2.2.2]{bib:Die-graph}), $G$ has a
2-factor. It will be helpful in our proof to pick a 2-factor with
special properties, namely one satisfying the condition in the
following result of Kaiser and \v{S}krekovski~\cite[Corollary
4.5]{bib:KS-cycles}:
\begin{theorem}[\cite{bib:KS-cycles}]\label{t:ks}
  Every cubic bridgeless graph contains a 2-factor whose edge set
  intersects each inclusionwise minimal edge-cut in $G$ of size 3 or
  4.
\end{theorem}

Among all 2-factors of $G$ satisfying the condition of
Theorem~\ref{t:ks}, choose a 2-factor $F$ with as many components as
possible. The following lemma will be used to rule out some of the
cases in the analysis found in Section~\ref{sec:chord}:

\begin{lemma}\label{l:split-cycle}
  Let $C$ be a cycle of $F$. If there exist vertex-disjoint cycles $D_1$
  and $D_2$ such that $V(C) = V(D_1)\cup V(D_2)$, then the following
  hold:
  \begin{enumerate}[\quad (i)]
  \item $2\leq \size{E(D_1,D_2)}\leq 4$,
  \item if the length of $D_1$ or $D_2$ equals 5, then
    $\size{E(D_1,D_2)}\leq 3$.
  \end{enumerate}
\end{lemma}
\begin{proof}
  Let $d = \size{E(D_1,D_2)}$. We prove (i). Clearly, $d\geq 2$ since
  at least two edges of $C$ join $D_1$ to $D_2$. Suppose that $d\geq
  5$. We claim that the 2-factor $F'$ obtained from $F$ by replacing
  $C$ with $D_1$ and $D_2$ satisfies the condition of
  Theorem~\ref{t:ks}. If not, then there is an inclusionwise minimal
  edge-cut $Y$ of $G$ of size 3 or 4 disjoint from $E(F')$. Since $Y$
  intersects $E(F)$, it must separate $D_1$ from $D_2$ and hence
  contain $E(D_1,D_2)$. But then $\size Y \geq 5$, a contradiction
  which shows that $F'$ satisfies the condition of
  Theorem~\ref{t:ks}. Having more components than $F$, it contradicts
  the choice of $F$. Thus, $d\leq 4$.

  (ii) Assume that $d = 4$ and that the length of, say, $D_1$ equals
  5. Let $F'$ be defined as in part (i). By the same argument,
  $E(D_1,D_2)$ is the unique inclusionwise minimal edge-cut of $G$
  disjoint from $E(F')$. Let $K_1$ be the component of $G-E(D_1,D_2)$
  containing $D_1$. Since $\bd{D_1}$ contains exactly one edge of
  $K_1$, this edge is a bridge in $G$, contradicting the assumption
  that $G$ is bridgeless.
\end{proof}

We fix some more notation used throughout the paper. Let $M$ be the
perfect matching complementary to $F$. If $u \in V(G)$, then $u'$
denotes the opposite endvertex of the edge of $M$ containing $u$. We
call $u'$ the \emph{mate} of $u$. We fix a reference orientation of
each cycle of $F$, and let $u_{+k}$ (where $k$ is a positive integer)
denote the vertex reached from $u$ by following $k$ consecutive edges
of $F$ in accordance with the fixed orientation. The symbol $u_{-k}$
is defined symmetrically. We write $u_+$ and $u_-$ for $u_{+1}$ and
$u_{-1}$. These vertices are referred to as the \emph{$F$-neighbours}
of $u$. 

We now describe \textbf{Algorithm 1}, an algorithm to construct a
random independent set $I$ in $G$. We will make use of a random
operation, which we define next. An independent set is said to be
\emph{maximum} if no other independent set has larger
cardinality. Given a set $X \subseteq V(G)$, we define $\Phi(X)
\subseteq X$ as follows:
\begin{enumerate}[\quad(a)]
\item if $F[X]$ is a path, then $\Phi(X)$ is either a maximum
  independent set of $F[X]$ or its complement in $X$, each with
  probability $1/2$,
\item if $F[X]$ is a cycle, then $\Phi(X)$ is a maximum independent
  set in $F[X]$, chosen uniformly at random,
\item if $F[X]$ is disconnected, then $\Phi(X)$ is the union of the
  sets $\Phi(X \cap V(K))$, where $K$ ranges over all components of
  $F[X]$.
\end{enumerate}

In \textbf{Phase 1} of the algorithm, we choose an orientation
$\vec\sigma$ of $M$ by directing each edge of $M$ independently at
random, choosing each direction with probability $1/2$. A vertex $u$
is \emph{active} (with respect to $\vec\sigma$) if $u$ is a head of
$\vec\sigma$, otherwise it is \emph{inactive}.

An \emph{active run} of $\vec\sigma$ is a maximal set $R$ of vertices such
that the induced subgraph $F[R]$ is connected and each vertex in $R$
is active. Thus, $F[R]$ is either a path or a cycle. We let
\begin{equation*}
  \sigma^1 = \bigcup_R \Phi(R),
\end{equation*}
where $R$ ranges over all active runs of $\vec\sigma$. The independent
set $I$ (which will be modified by subsequent phases of the algorithm
and eventually become its output) is defined as $\sigma^1$. The
vertices of $\sigma^1$ are referred to as those \emph{added in Phase
  1}. This terminology will be used for the later phases as well.

In \textbf{Phase 2}, we add to $I$ all the active vertices $u$ such
that each neighbour of $u$ is inactive. Observe that if an active run
consists of a single vertex $u$, then $u$ will be added to $I$ either
in Phase 1 or in Phase 2.

In \textbf{Phase 3}, we consider the set of all vertices of $G$ which
are not contained in $I$ and have no neighbour in $I$. We call such
vertices \emph{feasible}. Note that each feasible vertex must be
inactive. A \emph{feasible run} $R$ is defined analogously to an
active run, except that each vertex in $R$ is required to be feasible.

We define $\sigma^3 = \bigcup_R \Phi(R)$, where $R$ now ranges over
all feasible runs. All of the vertices of $\sigma^3$ are added to
$I$. 

In \textbf{Phase 4}, we add to $I$ all the feasible vertices with no
feasible neighbours. As with Phase 2, a vertex which forms a feasible
run by itself is certain to be added to $I$ either in Phase 3 or in
Phase 4.

When referring to the random independent set $I$ in
Sections~\ref{sec:events}--\ref{sec:chord}, we mean the set output
from Phase 4 of Algorithm 1. It will, however, turn out that this set
needs to be further adjusted in certain special situations. This
augmentation step will be performed in Phase 5, whose discussion we
defer to Section~\ref{sec:phase5}.

We represent the random choices made during an execution of
Algorithm~1 by the triple $\sigma = (\vec\sigma,\sigma^1,\sigma^3)$
which we call a \emph{situation}. Thus, the set $\Omega$ of all
situations is the sample space in our probabilistic scenario. As usual
for finite probabilistic spaces, an \emph{event} is any subset of
$\Omega$.

Note that if we know the situation $\sigma$ associated with a
particular run of Algorithm 1, we can determine the resulting
independent set $I = I(\sigma)$. We will say that an event $\Gamma
\subseteq \Omega$ \emph{forces} a vertex $u\in V(G)$ if $u$ is
included in $I(\sigma)$ for any situation $\sigma\in\Gamma$.


\section{Templates and diagrams}
\label{sec:events}

Throughout this and the subsequent sections, let $u$ be a fixed vertex
of $G$, and let $v = u'$. Furthermore, let $Z$ be the cycle of $F$
containing $u$. All cycles of $F$ are taken to have a preferred
orientation, which enables us to use notation such as $uZv$ for
subpaths of these cycles. 

We will analyze the probability of the event $u\in I(\sigma)$, where
$\sigma$ is a random situation produced by Algorithm 1. To this end,
we classify situations based on what they look like in the vicinity of
$u$.

A \emph{template} in $G$ is a 5-tuple $\Delta =
(\vec\Delta,\Delta^1,\Delta^{\bar 1},\Delta^3,\Delta^{\bar 3})$,
where:
\begin{itemize}
\item $\vec\Delta$ is an orientation of a subgraph of $M$,
\item $\Delta^1$ and $\Delta^{\bar 1}$ are disjoint sets of heads of
  $\vec\Delta$,
\item $\Delta^3$ and $\Delta^{\bar 3}$ are disjoint sets of tails of
  $\vec\Delta$.
\end{itemize}
We set $\Delta^* = \Delta^1\cup\Delta^{\bar
  1}\cup\Delta^3\cup\Delta^{\bar 3}$. The \emph{weight} of $\Delta$,
denoted by $w(\Delta)$, is defined as 
\begin{equation*}
  w(\Delta) = \size{E(\vec\Delta)} + \size{\Delta^*}.
\end{equation*}

A situation $\sigma = (\vec\sigma,\sigma^1,\sigma^3)$ \emph{weakly
  conforms to $\Delta$} if the following hold:
\begin{itemize}
\item $\vec\Delta\subseteq\vec\sigma$,
\item $\Delta^1 \subseteq \sigma^1$ and
\item $\Delta^{\bar 1} \cap \sigma^1 = \emptyset$.
\end{itemize}
If, in addition,
\begin{itemize}
\item $\Delta^3\subseteq\sigma^3$ and $\Delta^{\bar 3}\cap\sigma^3 =
  \emptyset$,
\end{itemize}
then we say that $\sigma$ \emph{conforms to $\Delta$}. The \emph{event
  defined by $\Delta$}, denoted by $\Gamma(\Delta)$, consists of all
situations conforming to $\Delta$.

By the above definition, we can think of $\Delta^1$ and $\Delta^{\bar
  1}$ as specifying which vertices must or must not be added to $I$ in
Phase 1. However, note that a vertex $u$ in an active run of length 1
will be added to $I$ in Phase 2 even if $u\in\Delta^{\bar
  1}$. Similarly, $\Delta^3$ and $\Delta^{\bar 3}$ specify which
vertices will or will not be added to $I$ in Phase 3, with an
analogous provision for feasible runs of length one.

To facilitate the discussion, we represent templates by pictorial
\emph{diagrams}. These usually show only the neighbourhood of the
distinguished vertex $u$, and the following conventions apply for a
diagram representing a template $\Delta$:
\begin{itemize}
\item the vertex $u$ is circled, solid and dotted lines represent
  edges and non-edges of $G$, respectively, dashed lines represent
  subpaths of $F$ (see Figure~\ref{fig:deficient}),
\item cycles and subpaths of $F$ are shown as circles and horizontal
  paths, respectively, and the edge $uv$ is vertical,
\item $u_-$ is shown to the left of $u$, while $v_-$ is shown to the
  right of $v$ (see Figure~\ref{fig:uv}),
\item the arcs of $\vec\Delta$ are shown with arrows,
\item the vertices in $\Delta^1$ ($\Delta^{\bar 1}$, $\Delta^3$,
  $\Delta^{\bar 3}$, respectively) are shown with a star (crossed
  star, triangle, crossed triangle, respectively),
\item only one endvertex of an arc may be shown (so an edge of $G$ may
  actually be represented by one or two arcs of the diagram), but the
  other endvertex may still be assigned one of the above symbols.
\end{itemize}
An arc with only one endvertex in a diagram is called an
\emph{outgoing} or an \emph{incoming} arc, depending on its
direction. A diagram is \emph{valid} in a graph $G$ if all of its
edges are present in $G$, and each edge of $G$ is given at most one
orientation in the diagram. Thus, a diagram is valid in $G$ if and
only if it determines a template in $G$. An event defined by a diagram
is \emph{valid} in $G$ if the diagram is valid in $G$.

\begin{figure}
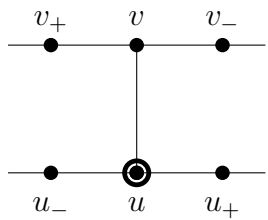

  \centering
  \fig{32}
  \caption{The location of neighbours of $u$ and $v$.}
  \label{fig:uv}
\end{figure}

A sample diagram is shown in Figure~\ref{fig:sample}. The
corresponding event (more precisely, the event given by the
corresponding template) consists of all situations
$(\vec\sigma,\sigma^1,\sigma^3)$ such that $v,v_+,u_{-2}$ and $\mate{u}{+}$
are heads of $\vec\sigma$, $\sigma^1$ includes $v_+$ and $u_{-2}$ but
does not include $\mate{u}{+}$, and $\sigma^3$ includes $u$.

\begin{figure}
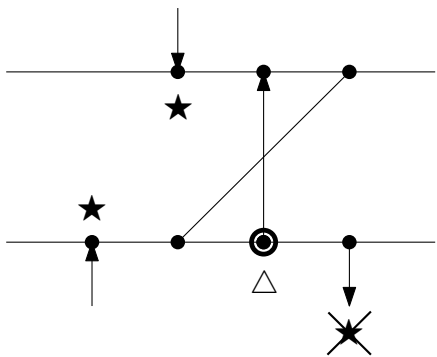

  \centering
  \fig1
  \caption{A diagram.}
  \label{fig:sample}
\end{figure}

Let us call a template $\Delta$ \emph{admissible} if
$\Delta^3\cup\Delta^{\bar 3}$ is either empty or contains only $u$,
and in the latter case, $u$ is feasible in any situation weakly
conforming to $\Delta$. All the templates we consider in this paper
will be admissible. Therefore, we state the subsequent definitions and
results in a form restricted to this case.

We will need to estimate the probability of an event defined by a
given template. If it were not for the sets $\Delta^1$, $\Delta^{\bar
  3}$ etc., this would be simple as the orientations of distinct edges
represent independent events. However, the events, say,
$u_1\in\Delta^1$ and $u_2\in\Delta^1$ (where $u_1$ and $u_2$ are
vertices) are in general not independent, and the amount of their
dependence is influenced by the orientations of certain edges of
$F$. To keep the dependence under control, we introduce the following
concept.

A \emph{sensitive pair} of a template $\Delta$ is an ordered pair
$(x,y)$ of vertices in $\Delta^1 \cup \Delta^{\bar 1}$, such that $x$
and $y$ are contained in the same cycle $W$ of $F$, the path $xWy$ has
no internal vertex in $\Delta^*$ and one of the following conditions
holds:
\begin{enumerate}[\quad(a)]
\item $x,y\in\Delta^1$ or $x,y\in\Delta^{\bar 1}$, $x\neq y$, the path
  $xWy$ has odd length and contains no tail of $\vec\Delta$,
\item $x\in\Delta^1$ and $y\in\Delta^{\bar 1}$ or vice versa, the path
  $xWy$ has even length and contains no tail of $\vec\Delta$,
\item $x=y\in\Delta^1$, $W$ is odd and contains no tail of
  $\vec\Delta$.
\end{enumerate}
Sensitive pairs of the form $(x,x)$ are referred to as
\emph{circular}, the other ones are \emph{linear}.

A sensitive pair $(x,y)$ is \emph{$k$-free} (where $k$ is a positive
integer) if $xWy$ contains at least $k$ vertices which are not heads
of $\vec\Delta$. Furthermore, any pair of vertices which is not
sensitive is considered $k$-free for any integer $k$. 

We define a number $q(\Delta)$ in the following way: If $u\in\Delta^3$
and $Z$ is an odd cycle, then $q(\Delta)$ is the probability that all
vertices of $Z$ are feasible with respect to a random situation from
$\Gamma(\Delta)$; otherwise, $q(\Delta)$ is defined as 0.

\begin{observation}\label{obs:q}
  Let $\Delta$ be a template in $G$. Then:
  \begin{enumerate}[\quad(i)]
  \item $q(\Delta) = 0$ if $u\notin\Delta^3$ or $Z$ contains a head of
    $\vec\Delta$ or $Z$ is even,
  \item $q(\Delta) \leq 1/2^t$ if $Z$ contains at least $t$ vertices
    which are not tails of $\vec\Delta$.
  \end{enumerate}
\end{observation}

The following lemma is a basic tool for estimating the probability of
an event given by a template.

\begin{lemma}\label{l:dep}
  Let $G$ be a graph and $\Delta$ an admissible template in $G$ such
  that:
  \begin{enumerate}[\quad(i)]
  \item $\Delta$ has $\ell$ linear sensitive pairs, the $i$-th of which
    is $x_i$-free ($i=1,\dots,\ell$), and
  \item $\Delta$ has $c$ circular sensitive pairs, the $i$-th of which
    is $y_i$-free ($i=1,\dots,c$).
  \end{enumerate}
  Then
  \begin{equation*}
    \prob{\Gamma(\Delta)} \geq \Bigl(1 - \sum_{i=1}^\ell
    \frac1{2^{x_i}} - \sum_{i=1}^c \frac1{5 \cdot 2^{y_i}} -
    \frac{q(\Delta)}5\Bigr) 
    \cdot \frac1{2^{w(\Delta)}}.
  \end{equation*}
\end{lemma}
\begin{proof}
  Consider a random situation $\sigma$. We need to estimate the
  probability that $\sigma$ conforms to $\Delta$. We begin by
  investigating the probability $P_1$ that $\sigma$ weakly conforms to
  $\Delta$.

  In Phase 1, the orientation $\vec{\sigma}$ is chosen by directing
  each edge of $M$ independently at random, each direction being
  chosen with probability $1/2$. Therefore, the probability that the
  orientation of each edge in the subgraph specified by $\vec{\Delta}$
  agrees with the orientation chosen at random is
  $(1/2)^{\size{E(\vec{\Delta})}}$.

  As noted above, the sets $\Delta^1,\Delta^{\bar 1},\Delta^3$ and
  $\Delta^{\bar 3}$ prescribe vertices to be added or not added in
  Phases 1 and 3 of the algorithm.

  Suppose, for now, that every active run $R$ has $\size{R \cap
    (\Delta^1 \cup \Delta^{\bar 1})}=1$ and is either a path or an
  even cycle. Then a given vertex in $\Delta^1$ is added in Phase 1
  with probability $1/2$. Likewise, a given vertex in $\Delta^{\bar
    1}$ is not added in Phase 1 with probability $1/2$. Indeed, $R$
  has either one or two maximum independent sets and $\Phi(R)$ chooses
  either between the maximum independent set and its complement or
  between the two maximum independent sets.

  There are $\size{\Delta^1}$ vertices required to be added in Phase 1
  and $\size{\Delta^{\bar 1}}$ vertices required to not be added in
  Phase 1. These events are independent each with probability $1/2$,
  giving the resultant probability
  \begin{equation}\label{eq:phase-1}
    \prob{\Delta^1 \subseteq \sigma_1, \Delta^{\bar 1} \cap \sigma_1= \emptyset} = 
    \Bigl(\frac12\Bigr)^{\size{\Delta^1}+\size{\Delta^{\bar 1}}}.
  \end{equation}
  The probability $P_1$ is obtained by multiplying~\eqref{eq:phase-1}
  by $(1/2)^{\size{E(\vec{\Delta})}}$.

  We now assess the probability that $\sigma$ conforms to $\Delta$
  under the assumption that it conforms weakly. If $\Delta^3\cup
  \Delta^{\bar 3}$ is empty, the probability is 1 for trivial
  reasons. Otherwise, the admissibility of $\Delta$ implies that
  $\Delta^3\cup \Delta^{\bar 3} = \Setx u$ and $u$ is feasible with
  respect to $\sigma$. Let $R$ be the feasible run containing
  $u$. Suppose that $R$ is a path or an even cycle. Then if $u \in
  \Delta^3$, it is added in Phase 3 with probability $1/2$, and if $u
  \in \Delta^{\bar 3}$, it is not added in Phase 3 with probability
  $1/2$. Since $u$ is the only vertex allowed in $\Delta^3\cup
  \Delta^{\bar 3}$, we obtain
  \begin{equation*}
    \prob{\Delta^3 \subseteq \sigma^3, \Delta^{\bar 3} \cap \sigma^3= \emptyset} =
    \begin{cases}
      \frac12 & \text{ if $u\in \Delta^3\cup \Delta^{\bar 3}$,}\\
      1 & \text{ otherwise}.
    \end{cases}
  \end{equation*}
  The assumption that $u$ is feasible whenever $\sigma$ weakly
  conforms to $\Delta$ and $\Delta^3\cup \Delta^{\bar 3} = \Setx{u}$
  implies that the addition of $u$ to $\sigma^3$ is independent of the
  preceding random choices.

  Note that we can relax the assumptions above to allow, for instance,
  $\size{R\cap \Delta^1}\geq 1$, provided that the vertices of
  $\Delta^1$ are appropriately spaced. Suppose that $x,y$ are in the
  same component $W$ of $F$ and all vertices of $xWy$ are active after
  the choice of orientations in Phase 1. Let $R$ be the active run $R$
  containing $xWy$.

  Observe that if $d_W(x,y)$ is even, $\Phi$ will choose both $x$ and
  $y$ with probability $1/2$ for addition to $I$ in Phase 1, an
  increase compared to the probability $1/4$ if they are in different
  active runs. On the other hand, if $d_W(x,y)$ is odd, then the
  probability of adding both $x$ and $y$ is zero as $x$ and $y$ cannot
  both be contained in $\Phi(R)$. Thus, if $x,y\in\Delta^1$ and
  $d_W(x,y)$ is odd, then $x$ and $y$ must be in distinct active runs
  with respect to any situation conforming to $\Delta$. As a result,
  we will in general get a lower value for the probability in
  \eqref{eq:phase-1}; the estimate will depend on the sensitive pairs
  involved in $\Delta$.

  \begin{figure}
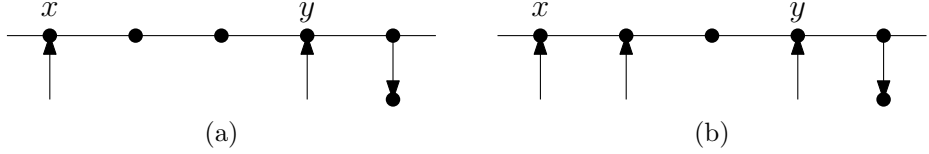

    \centering
    \hf\sfig{73}{}\hf\sfig{74}{}\hf
    \caption{The probability that $x$ and $y$ are in distinct active
      runs in a conforming random situation is: (a) $3/4$, (b) $1/2$.}
    \label{fig:sensitive}
  \end{figure}

  Let $(x,y)$ be a $k$-free sensitive pair contained in a cycle $W$ of
  $F$, and let the internal vertices of $xWy$ which are not heads of
  $\Delta$ be denoted by $x_1,\dots,x_k$. Suppose that $(x,y)$ is of
  type (a); say, $x,y\in \Delta^1$. The active runs of $x$ and $y$
  with respect to $\sigma$ will be separated if we require that at
  least one of $x_1,x_2,...,x_k$ is the tail of an arc of
  $\vec{\sigma}$, which happens with probability $1-(1/2)^k$. The same
  computation applies to a sensitive pair of type (b).

  Now suppose that $(x,x)$ is sensitive of type (c), i.e., $x$ is the
  only member of $\Delta^1$ belonging to an odd cycle $W$ of length
  $\ell$. If some vertex of $W$ is the tail of an arc of $\vec\sigma$,
  then $x$ will be added in Phase 1 with probability $1/2$ as
  usual. It can happen, however (with probability $1/2^{\ell-1}$),
  that all the vertices of $W$ are heads in $\vec\sigma$, in which
  case $\Phi(V(W))$ is one of $\ell$ maximum independent sets in
  $W$. If this happens, $x$ will be added to $I$ with probability $1/2
  \cdot (\ell-1)/\ell\geq 2/5$ rather than $1/2$; this results in a
  reduction in $\prob{\Gamma(\Delta)}$ of at most $1/5 \cdot
  1/2^{\ell-1}\cdot 1/2^{w(\Delta)}$.

  Finally, let us consider the situation where $u\in\Delta^3$ and the
  feasible run containing $u$ is cyclic, that is, the case where every
  vertex in $C_u$ is feasible. If $Z$ is even, then this has no effect
  as $u$ is still added in Phase 3 with probability $1/2$. If $Z$ is
  odd, then $u$ is added in Phase 3 with probability at least $2/5$
  instead. Thus, if the probability of all vertices in $C_u$ being
  feasible is $q(\Delta)$, then the resultant loss of probability from
  $\prob{\Gamma(\Delta)}$ is at most $q(\Delta)/(5\cdot
  2^{w(\Delta)})$.

  Putting all this together gives:
  \begin{equation*}
    \prob{\Gamma(\Delta)} \geq
    \left( 1 - \sum_{i=1}^{\ell}\frac{1}{2^{x_i}} - 
      \sum_{i=1}^{c}\frac{1}{5\cdot 2^{y_i}} - 
      \frac{q(\Delta)}{5} \right) 
    \cdot \frac{1}{2^{w(\Delta)}}
  \end{equation*}
  as required.
\end{proof}

We remark that by a careful analysis of the template in question, it
is sometimes possible to obtain a bound better than that given by
Lemma~\ref{l:dep}; however, the latter bound will usually be
sufficient for our purposes.

A template without any sensitive pairs is called \emph{weakly
  regular}. If a weakly regular template $\Delta$ has $q(\Delta) = 0$,
then it is \emph{regular}. By Lemma~\ref{l:dep}, if $\Delta$ is a
regular template, then $\prob{\Gamma(\Delta)} \geq
1/2^{w(\Delta)}$. When using Lemma~\ref{l:dep} in this way, we will
usually just state that the template in question is regular and give
its weight, and leave the straightforward verification to the reader.

The analysis is often more involved if sensitive pairs are present. To
allow for a brief description of a template $\Delta$, we say that
$\Delta$ is \emph{covered (in $G$) by} ordered pairs of vertices
$(x_i,y_i)$, where $i=1,\dots,k$, if every sensitive pair of $\Delta$
is of the form $(x_i,y_i)$ for some $i$. In most cases, our
information on the edge set of $G$ will only be partial; although we
will not be able to tell for sure whether any given pair of vertices
is sensitive, we will be able to restrict the set of possibly
sensitive pairs.

For brevity, we also use $\pair x y \ell$ to denote an $\ell$-free
pair of vertices $(x,y)$. Thus, we may write, for instance, that a
template $\Delta$ is covered by pairs $\pair x y 2$ and $\pair z z
4$. By Lemma~\ref{l:dep}, we then have $\prob{\Gamma(\Delta}) \geq
1/2^{w(\Delta)} \cdot (1 - 1/4 - 1/80)$.

In some cases, the structure of $G$ may make some of the symbols in a
diagram redundant. For instance, consider the diagram in
Figure~\ref{fig:removable}(a) and let $R_1$ be the event corresponding
to the associated template $\Delta_1$. Since the weight of $\Delta_1$
is 4, Lemma~\ref{l:dep} implies a lower bound for $\prob{R_1}$ which
is slightly below $1/16$. However, if we happen to know that the mate
of $u_+$ is $v_-$, then we can remove the symbol at $\mate{u}{+}$; the
resulting diagram encodes the same event and comes with a better bound
of $1/8$. We will describe this situation by saying that the symbol at
$\mate{u}{+}$ in the diagram for $\Delta_1$ is \emph{removable} (under the
assumption that $\mate{u}{+}=v_-$).

\begin{figure}
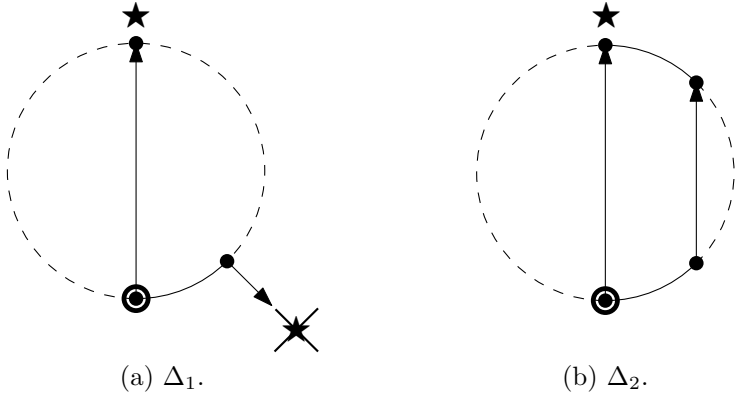

  \centering
  \sfigdef{85}
  \hf\sfigtop{85}{$\Delta_1$.}\hf\sfigtop{86}{$\Delta_2$.}\hf
  \caption{The symbol at $\mate{u}{+}$ in the diagram defining the template
    $\Delta_1$ becomes removable if we add the assumption that
    $\mate{u}{+}=v_-$.}
  \label{fig:removable}
\end{figure}

We extend the terminology used for templates to events defined by
templates. Suppose that $\Delta$ is a template in $G$. The properties
of $\Gamma(\Delta)$ simply reflect those of $\Delta$. Thus, we say
that the event $\Gamma(\Delta)$ is \emph{regular} (\emph{weakly
  regular}) if $\Delta$ is regular (weakly regular), and we set
$q(\Gamma(\Delta)) = q(\Delta)$. A pair of vertices is said to be
\emph{$k$-free} for $\Gamma(\Delta)$ if it is $k$-free for $\Delta$;
$\Gamma(\Delta)$ is \emph{covered by} a set of pairs of vertices if
$\Delta$ is.


\section{Events forcing a vertex}
\label{sec:forcing}

In this section, we build up a repertoire of events forcing the
distinguished vertex $u$. (Recall that $u$ is forced by an event
$\Gamma$ if $u$ is contained in $I(\sigma)$ for every situation
$\sigma\in\Gamma$.) In our analysis, we will distinguish various cases
based on the local structure of $G$ and show that in each case, the
total probability of these events (and thus the probability that $u\in
I$) is large enough.

Suppose first that $\sigma$ is a situation for which $u$ is active. By
the description of Algorithm 1, we will have $u\in I$ if either both
$u_+$ and $u_-$ are inactive, or $u\in\sigma^1$. Thus, each of the
templates $E^0,E^-,E^+,E^\pm$ represented by the diagrams in
Figure~\ref{fig:local1} defines an event which forces $u$. These
events (which will be denoted by the same symbols as the templates,
e.g., $E^0$) are pairwise disjoint. Observe that by the assumption
that $G$ is simple and triangle-free, each of the diagrams is valid in
$G$.

\begin{figure}
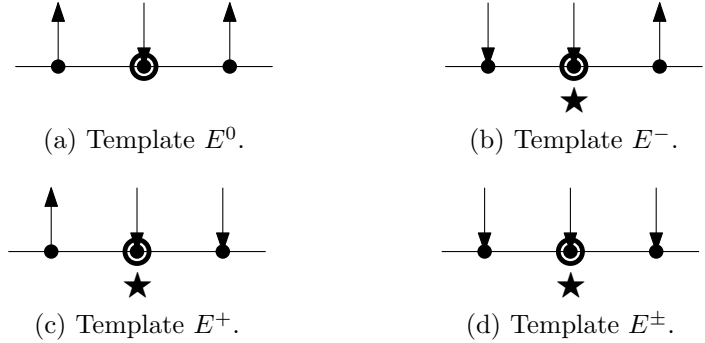

  \centering\hf
  \sfigdef4
  \sfigtop3{Template $E^0$.}\hf
  \sfigtop4{Template $E^-$.}\hf\\\hf
  \sfig5{Template $E^+$.}\hf
  \sfig6{Template $E^\pm$.}\hf
  \caption{Some templates defining events which force $u$.}
  \label{fig:local1}
\end{figure}

It is not difficult to estimate the probabilities of these events.
The event $E^0$ is regular of weight 3, so $\prob{E^0} \geq 1/8 =
32/256$ by Lemma~\ref{l:dep}. Similarly, $E^+$ and $E^-$ are regular
of weight 4 and have probability at least $16/256$ each. The weakly
regular event $E^\pm$ has weight 4 and the only potentially sensitive
pair is $(u,u)$. If the pair is sensitive, the length of $Z$ must be
odd and hence at least 5; thus, the pair is 2-free. By
Lemma~\ref{l:dep},
\begin{align*}
  \prob{E^\pm} \geq \frac1{16} \cdot \frac{19}{20} = 15.2/256.
\end{align*}
Note that if $Z$ has a chord (for instance, $uv$), then $E^\pm$ is
actually regular, which improves the above estimate to $16/256$.

By the above,
\begin{equation*}
  \prob{E^0 \cup E^+ \cup E^- \cup E^\pm} \geq \frac{32 + 16 + 16 +
    15.2}{256} = \frac{79.2}{256}.
\end{equation*}
These events cover most of the situations where $u\in I$. To prove
Theorem~\ref{t:main}, we will need to find other situations which also
force $u$ and their total probability is at least about one tenth of
the above. Although this number is much smaller, finding the required
events turns out to be a more difficult task.

Since Figure~\ref{fig:local1} exhausts all the possibilities where $u$
is active, we now turn to the situations where $u$ is inactive.

Assume an event forces $u$ although $u$ is inactive. We find that if
$u_-$ is active, then $u_{-2}$ must be added in Phase 1. If $u_-$ is
inactive, then there are several configurations which allow $u$ to be
forced, for instance if $u$ is added in Phase 3. However, the result
also depends on the configurations around $u_+$ and $v$. We will
express the events forcing $u$ as combinations of certain `primitive'
events.

Let us begin by defining templates $A,B,C_1,C_2,C_3$ (so called
\emph{left templates}). We remind the reader that the vertex $v$ is
the mate of $u$. Diagrams corresponding to the templates are given in
Figure~\ref{fig:left}:
\begin{center}
  \begin{tabular}{c|c|c}
    template & heads of $\vec\sigma$ & other conditions\\\hline
    $A$ & $v,u_-,u_{-2}$ & $u_{-2}\in\sigma^1$\\
    $B$ & $v,\mate{u}{-}$ & $u\in\sigma^3$\\
    $C_1$ & $v,\mate{u}{-}$ & $u\notin\sigma^3, \mate{u}{-}\in\sigma^1$\\
    $C_2$ & $v,\mate{u}{-},u_{-2}$ & $u\notin\sigma^3, \mate{u}{-}\notin\sigma^1, u_{-2}\in\sigma^1$\\
    $C_3$ & $v,\mate{u}{-},u_{-2},\mate{u}{-3}$ & $u\notin\sigma^3, \mate{u}{-}\notin\sigma^1, u_{-2}\notin\sigma^1$
  \end{tabular}  
\end{center}
In addition, for $P\in\Setx{A,B,C_1,C_2,C_3}$, the template $P^*$ is
obtained by exchanging all `$-$' signs for `$+$' in this
description. These are called \emph{right templates}. In our diagrams,
templates such as $A$ or $C_1$ restrict the situation to the left of
$u$, while templates such as $A^*$ or $C_1^*$ restrict the situation
to the right.

\begin{figure}
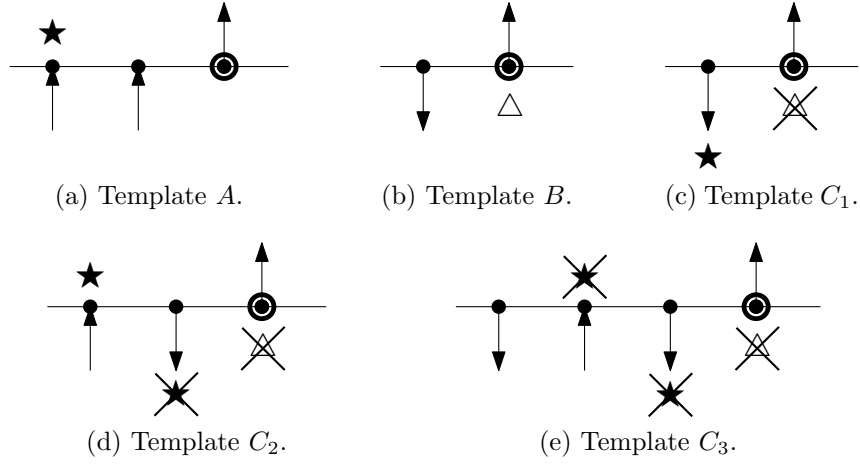

  \centering\hf
  \sfigdef9
  \sfigtop7{Template $A$.}\hf
  \sfigtop8{Template $B$.}\hf
  \sfigtop9{Template $C_1$.}\hf\\\hf
  \sfigdef{10}
  \sfigtop{10}{Template $C_2$.}\hf
  \sfigtop{11}{Template $C_3$.}\hf
  \caption{Left templates.}
  \label{fig:left}
\end{figure}

We also need primitive templates related to $v$ and its neighbourhood
(\emph{upper templates}), for the configuration here is also
relevant. These are simpler (see Figure~\ref{fig:upper}):
\begin{center}
  \begin{tabular}{c|c|c}
    template & heads of $\vec\sigma$ & other conditions\\\hline
    $D^-$ & $v,v_-,\mate{v}{+}$ & $v_-\in\sigma^1$\\
    $D^0$ & $v,v_-,v_+$ & $v\notin\sigma^1$\\
    $D^+$ & $v,\mate{v}{-},v_+$ & $v_+\in\sigma^1$
  \end{tabular}  
\end{center}

\begin{figure}
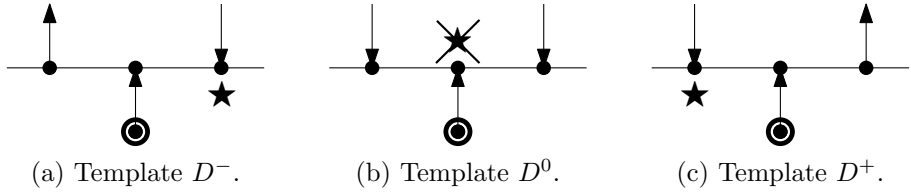

  \centering\hf
  \sfigdef{14}
  \sfigtop{14}{Template $D^-$.}\hf
  \sfigtop{13}{Template $D^0$.}\hf 
  \sfigtop{12}{Template $D^+$.}\hf
  \caption{Upper templates.}
  \label{fig:upper}
\end{figure}

We can finally define the templates obtained from the left, right and
upper events as their combinations. More precisely, for
$P,Q\in\Setx{A,B,C_1,C_2,C_3}$ and $R\in\Setx{D^-,D^0,D^+}$, we define
$PQR$ to be the template $\Delta$ such that
\begin{align*}
  \vec\Delta &= \vec P \cup \vec{Q^*} \cup \vec R,\\
  \Delta^1 &= P^1 \cup (Q^*)^1 \cup R^1,
\end{align*}
and so on for the other constituents of the template. The same symbol
$PQR$ will be used for the event defined by the template. If the
result is not a legitimate template (for instance, because an edge is
assigned both directions, or because $u$ is required to be both in
$\Delta^3$ and $\Delta^{\bar 3}$), then the event is an empty one and
is said to be \emph{invalid}, just as if it were defined by an invalid
diagram.

Let $\Sigma$ be the set of all valid events $PQR$ given by the above
templates. Thus, $\Sigma$ includes, e.g., the events $AAD^0$ or
$BC_1D^+$. However, some of them (such as $BC_1D^+$) may be invalid,
and the probability of others will in general depend on the structure
of $G$. We will examine this dependence in detail in the following
section. It is not hard to check (using the description of Algorithm
1) that each of the valid events in $\Sigma$ forces $u$ and also that
each of them is given by an admissible template, as defined in
Section~\ref{sec:events}.


\section{Analysis: $uv$ is not a chord}
\label{sec:nochord}

We are going to use the setup of the preceding sections to prove
Theorem~\ref{t:main} for a cubic bridgeless graph $G$. Recall that $v$
denotes the vertex $u'$ and $Z$ denotes the cycle of $F$ containing
$u$. If we can show that $\prob{u\in I} \geq 11/32$, then by
Lemma~\ref{l:fraccol}, $\chi_f(G) \leq 32/11$ as required. Thus, our
task will be accomplished if we can present disjoint events forcing
the fixed vertex $u$ whose probabilities sum up to at least $11/32 =
88/256$. It will turn out that this is not always possible, which will
make it necessary to use a compensation step discussed in
Section~\ref{sec:phase5}.

In this section, we begin with the case where $v$ is contained in a
cycle $C_v \neq Z$ of $F$ (that is, $uv$ is not a chord of $Z$). We
define a number $\eps(u)$ as follows:
\begin{equation*}
  \eps(u) = \begin{cases}
    1 & \text{if $uv$ is contained in a 4-cycle,}\\
    0 & \text{if $u$ has no $F$-neighbour contained in a 4-cycle
      intersecting $C_v$,}\\
    -1 & \text{otherwise}.
  \end{cases}
\end{equation*}
The vertices with $\eps(u)=-1$ will be called \emph{deficient of type
  0}.

The end of each case in the proof of the following lemma is marked by
$\blacktriangle$.

\begin{lemma}
  \label{l:nochord}
  If $uv$ is not a chord of $Z$, then 
  \begin{equation*}
    \prob{u\in I} \geq \frac{88+\eps(u)}{256}.
  \end{equation*}
\end{lemma}
\begin{proof}
  As observed in Section~\ref{sec:forcing}, the probability of the
  event $E^0 \cup E^-\cup E^+$ is at least $64/256$. For the event
  $E^\pm$, we only get the estimate $\prob{E^\pm} \geq 15.2/256$,
  which yields a total of $79.2/256$.

  \begin{xcase}{The edge $uv$ is contained in two
      4-cycles.}\label{case:nochord-two-c4}%
    Consider the event $BBD^0$ of weight 5 (see the diagram in
    Figure~\ref{fig:nochord-c4}(a)). We claim that $\prob{BBD^0} \geq
    8/256$. Note that for any situation $\sigma \in BBD^0$, at least
    one of the vertices $v_-$, $v_+$ is added to $I$ in Phase 1. It
    follows that for any such situation, $u_-$ or $u_+$ is
    infeasible. Thus, $q(BBD^0) = 0$, and by Lemma~\ref{l:dep},
    \begin{equation*}
      \prob{BBD^0} \geq \frac1{2^5} = \frac8{256}.
    \end{equation*}

    Next, we use the event $ABD^-$ of weight 7 (see
    Figure~\ref{fig:nochord-c4}(b)). Since $u_{-2}\in\sigma^1$ for any
    situation $\sigma\in ABD^-$, it is infeasible and hence
    $q(ABD^-)=0$. Furthermore, $ABD^-$ contains no sensitive pair and
    thus it is regular. Lemma~\ref{l:dep} implies that $\prob{ABD^-}
    \geq 2/256$. This shows that
    \begin{equation*}
      \prob{u\in I} \geq 89.2/256. 
    \end{equation*}
    We remark that a further contribution of $2/256$ could be obtained
    from the event $AC_1D^-$, but it will not be necessary.
  \end{xcase}

  \begin{figure}
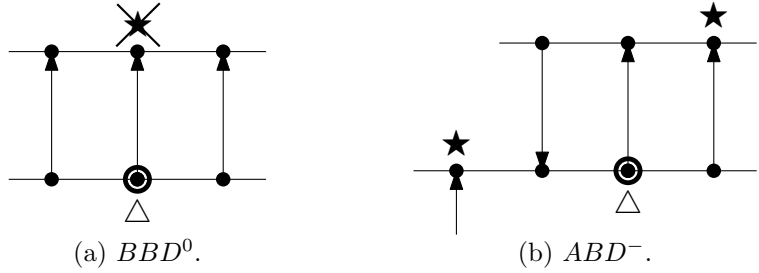

    \centering
    \sfigdef{40}
    \hf\sfigtop{40}{$BBD^0$.}\hf\sfigtop{61}{$ABD^-$.}\hf
    \caption{The events used in Case~\ref{case:nochord-two-c4} of the proof of
      Lemma~\ref{l:nochord}.}
    \label{fig:nochord-c4}
  \end{figure}

  \begin{xcase}{$uv$ is contained in one
      4-cycle.}\label{case:nochord-one-c4}%
    We may assume that $u_+$ is adjacent to $v_-$. From
    Figure~\ref{fig:nochord-one-c4}(a), we see that the event $BBD^0$
    is weakly regular; we will estimate $q(BBD^0)$. Let $\sigma$ be a
    random situation from $BBD^0$. If $C_v$ is even, then
    $v_-\in\sigma^1$, which makes $u_+$ infeasible, so $q(BBD^0) =
    0$. Assume then that $C_v$ is odd; since $G$ is triangle-free, the
    length of $C_v$ is at least 5. Thus it contains at least two
    vertices other than $v,v_-,v_+$; consequently, the probability
    that all the vertices of $C_v$ are active is at most $1/4$. If all
    the vertices of $C_v$ are active, then $v_-\in\sigma^1$ (and hence
    $u_+$ is infeasible) with probability at least $2/5$. It follows
    that
    \begin{equation*}
      q(BBD^0) \leq \cprob{v_-\notin\sigma^1}{\sigma\in BBD^0} \leq \frac1{10}.
    \end{equation*}
    By Lemma~\ref{l:dep}, $\prob{BBD^0} \geq 98/100 \cdot 1/64 =
    3.92/256$.

    Consider the weakly regular event $BBD^-$
    (Figure~\ref{fig:nochord-one-c4}(b)). Observe first that the event is
    valid in $G$ as $u_-$ and $v_+$ are not neighbours. Since $u_+$ is
    infeasible with respect to any situation from $BBD^-$, we have
    $q(BBD^-) = 0$ and so $BBD^-$ is regular. Lemma~\ref{l:dep}
    implies that $\prob{BBD^-} \geq 4/256$.

    Finally, consider the events $ABD^-$ and $AC_1D^-$
    (Figure~\ref{fig:nochord-one-c4}(c) and (d)); note that the only
    difference between them is that for $\sigma\in ABD^-$,
    $u\in\sigma^3$, whereas for $\sigma\in AC_1D^-$ it is the
    opposite. Both events, however, force $u$. Observe that their
    validity does not depend on whether $u_{-2}$ and $v_+$ are
    neighbours: even if they are, the diagram prescribes consistent
    orientations at both ends of the edge $u_{-2}v_+$. The events are
    regular of weight 8, and thus $\prob{ABD^- \cup AC_1D^-} \geq
    2/256$. This proves that $\prob{u\in I} > 89.1/256$.
  \end{xcase}
  \begin{figure}
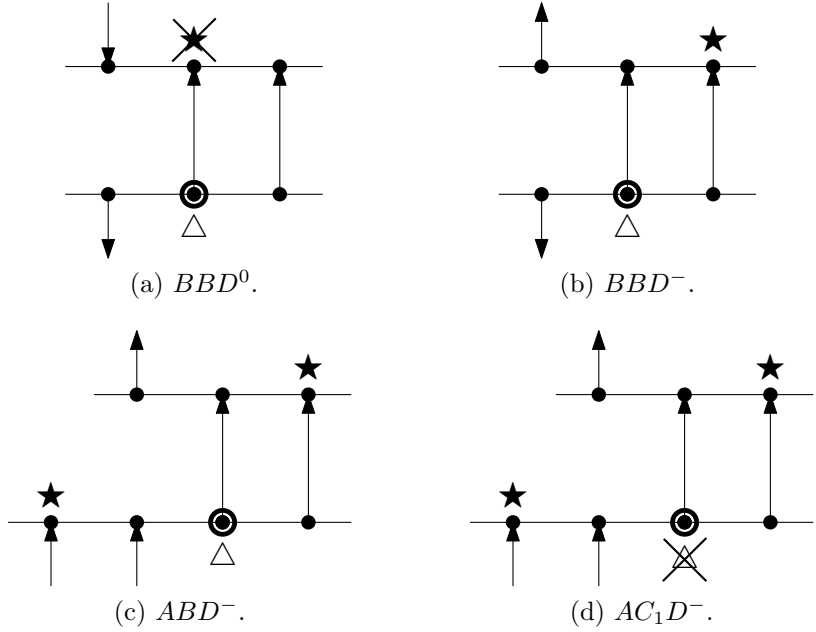

    \centering
    \sfigdef{65}
    \hf\sfigtop{62}{$BBD^0$.}\hf\sfigtop{63}{$BBD^-$.}\hf\\
    \hf\sfigtop{64}{$ABD^-$.}\hf\sfigtop{65}{$AC_1D^-$.}\hf
    \caption{The events used in Case~\ref{case:nochord-one-c4} of the proof of
      Lemma~\ref{l:nochord}.}
    \label{fig:nochord-one-c4}
  \end{figure}

  Having dealt with the above cases, we may now assume that the set
  $\Setx{u_-,u_+,v_-,v_+}$ is independent.

  \begin{xcase}{$M$ includes the edges $u_{-2}v_+$ and
      $u_{+2}v_-$.}\label{case:nochord-ll-rr}%
    The event $BBD^+$ (Figure~\ref{fig:nochord-ll-rr}(a)) is regular
    of weight 7; thus, $\prob{BBD^+} \geq 2/256$. Similarly,
    $\prob{BBD^-}\geq 2/256$. We also have $\prob{BBD^0} \geq 2/256$
    since $v_+$ and $v_-$ have mates on $Z$, ensuring that one of the
    vertices of $Z$ is infeasible and thus $q(BBD^0) =
    0$. Furthermore, $\prob{ABD^-\cup BAD^+} \geq 2/256$ by
    Lemma~\ref{l:dep}.
    
    We may assume that $\mate{u}{-} \neq v_{+2}$ and $\mate{u}{+} \neq v_{-2}$, for
    otherwise $u$ has a neighbour contained in a 4-cycle and
    $\eps(u)=-1$. In that case, the bound $\prob{u\in I} \geq
    87.2/256$, proved so far, would be sufficient.

    If $u_+$ or $u_-$ have a mate on $Z$, then $E^\pm$ is regular and
    hence $\prob{E^\pm} = 16/256$. This adds further $0.8/256$ to
    $\prob{u\in I}$, making it reach $88/256$, which is
    sufficient. Thus, we may assume that $\mate{u}{-}$ and $\mate{u}{+}$ are not
    contained in $Z$.

    Consider the event $C_1AD^+$ given by the diagram in
    Figure~\ref{fig:nochord-ll-rr}(b). Since this is the first time
    that the analysis of its probability involves a sensitive pair, we
    explain it in full detail. Assume that there exists a sensitive
    pair for this event. The only vertices which can be included in
    the pair are $\mate{u}{-}$, $u_{+2}$ and $v_+$. None of $(u_{+2},u_{+2})$
    and $(v_+,v_+)$ is a circular sensitive pair, since both $Z$ and
    $C_v$ contain a tail in $C_1AD^+$ ($u_-$ and $v_-$,
    respectively). Hence, the only possible circular sensitive pair is
    $(\mate{u}{-},\mate{u}{-})$. As for linear sensitive pairs, the only possibility
    is $(v_+,\mate{u}{-})$: the vertex $u_{+2}$ is ruled out since none of
    $\mate{u}{-}$ and $v_+$ is contained in $Z$, and the pair $(\mate{u}{-},v_+)$
    cannot be sensitive as $v_-$ is a tail in $C_1AD^+$. (Note that
    the sensitivity of a pair depends on the order of the vertices in
    the pair.) Summarizing, the sensitive pair is $(\mate{u}{-},\mate{u}{-})$ or
    $(v_+,\mate{u}{-})$, and it is clear that not both pairs can be sensitive
    at the same time.

    If $(\mate{u}{-},\mate{u}{-})$ is sensitive, then the cycle of $F$ containing
    $\mate{u}{-}$ contains at least four vertices which are not heads in
    $C_1AD^+$. Consequently, the pair $(\mate{u}{-},\mate{u}{-})$ is 4-free, and
    Lemma~\ref{l:dep} implies $\prob{C_1AD^+} \geq 79/80 \cdot 0.5/256
    > 0.49/256$. 

    On the other hand, if $(v_+,\mate{u}{-})$ is sensitive, we know that
    $d_{C_v}(v_+,\mate{u}{-})$ is odd, and our assumption that $\mate{u}{-} \neq
    v_{+2}$ implies that the pair $(v_+,\mate{u}{-})$ is 2-free. By
    Lemma~\ref{l:dep}, $\prob{C_1AD^+} \geq 3/4 \cdot 0.5/256 =
    0.375/256$. As this estimate is weaker than the preceding one,
    $C_1AD^+$ is guaranteed to have probability at least $0.375/256$.
    Symmetrically, $\prob{AC_1D^-} \geq 0.375/256$.

    So far, we have accumulated a probability of $87.95/256$. The
    missing bit can be supplied by the event $C_1C_2D^+$ of weight 10
    (Figure~\ref{fig:nochord-ll-rr}(c)). Since $u_-$ and $u_+$ do not
    have mates on $Z$, any sensitive pair will involve only the
    vertices $\mate{u}{-}$, $\mate{u}{+}$ and $v_+$, and it is not hard to check
    that there will be at most two such pairs. Since $\mate{u}{-} \neq
    v_{+2}$, each of these pairs is 1-free. If one of them is 2-free,
    then $\prob{C_1C_2D^+} \geq 1/4\cdot 0.25/256 > 0.06/256$ by
    Lemma~\ref{l:dep}, which is more than the amount missing to
    $88/256$. 

    We may thus assume that none of these pairs is 2-free. This
    implies that $(v_+,\mate{u}{-})$ is not a sensitive pair, as $d_{C_v}$
    would have to be odd and strictly between 1 and 3. Thus, there are
    only two possibilities: (a) $C_1C_2D^+$ is covered by
    $(\mate{u}{-},\mate{u}{+})$ and $(\mate{u}{+},\mate{u}{-})$, or (b) it is covered by
    $(v_+,\mate{u}{+})$ and $(\mate{u}{+},\mate{u}{-})$. The former case corresponds to
    $\mate{u}{+}$ and $\mate{u}{-}$ being contained in a cycle $W$ of $F$ of length
    4, which is impossible by the choice of $F$. In the latter case,
    $\mate{u}{+}$ and $\mate{u}{-}$ are contained in $C_v$; in fact, $\mate{u}{+} = v_{+3}$
    and $\mate{u}{-} = v_{+5}$. Although Lemma~\ref{l:dep} does not give us a
    nonzero bound for $\prob{C_1C_2D^+}$, we can get one by exploiting
    the fact that $G$ is triangle-free. Since $v_{+2}v_{+4}\notin
    E(M)$, the probability that both $v_{+2}$ and $v_{+4}$ are tails
    with respect to the random situation $\sigma$ is $1/4$, and these
    events are independent of orientations of the other edges of
    $G$. Thus, the probability that $\sigma$ weakly conforms to the
    template for $C_1C_2D^+$ and $v_{+2},v_{+4}$ are tails is $1/2^7 =
    2/256$. Under this condition, $\sigma$ will conform to the
    template with probability $1/2^5$ (a factor $1/2$ for each symbol
    in the diagram). Consequently, $\prob{C_1C_2D^+} > 0.06/256$,
    again a sufficient amount.
  \end{xcase}

  \begin{figure}
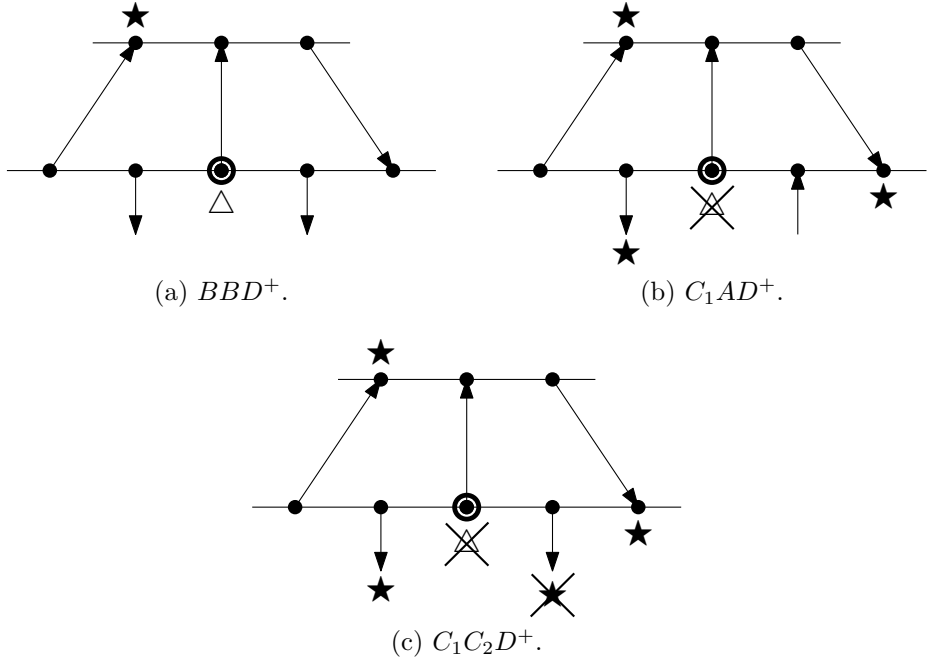

    \centering
    \sfigdef{67}
    \hf\sfigtop{66}{$BBD^+$.}\hf\sfigtop{67}{$C_1AD^+$.}\hf\\
    \hf\sfig{68}{$C_1C_2D^+$.}\hf
    \caption{Some of the events used in Case~\ref{case:nochord-ll-rr}
      of the proof of Lemma~\ref{l:nochord}.}
    \label{fig:nochord-ll-rr}
  \end{figure}

  \begin{xcase}{$M$ includes the edge $u_{-2}v_+$ but not
      $u_{+2}v_-$.}\label{case:nochord-ll}%
    As in the previous case, $\prob{BBD^-} \geq 2/256$. Consider the
    weakly regular event $BBD^+$ (Figure~\ref{fig:nochord-ll}(a)). Since
    $\mate{u}{-2} = v_+\in \sigma^1$ for any $\sigma\in BBD^+$, we have
    $q(BBD^+) = 0$. By Lemma~\ref{l:dep}, $\prob{BBD^+} \geq 2/256$.

    The event $BBD^0$ is also weakly regular, and it is not hard to
    see that $q(BBD^0) \leq 1/10$ (using the fact that the length of
    $C_v$ is at least 5). Lemma~\ref{l:dep} implies that $\prob{BBD^0}
    \geq 98/100 \cdot 2/256 = 1.96/256$.
    
    Each of the events $BAD^0$ (Figure~\ref{fig:nochord-ll}(b)),
    $BAD^+$ and $BAD^-$ is regular and has weight 9. By
    Lemma~\ref{l:dep}, it has probability at least
    $0.5/256$. Furthermore, $\prob{ABD^-} \geq 1/256$, also by
    regularity. So far, we have shown that $\prob{u\in I}\geq
    87.66/256$. As in the previous case, this enables us to assume
    that $\mate{u}{-}$ and $\mate{u}{+}$ are not vertices of $Z$. Furthermore, it
    may be assumed that $\mate{u}{-} \neq v_{+2}$, for otherwise $\eps(u) =
    -1$ and the current estimate on $\prob{u\in I}$ is sufficient.

    If $M$ includes the edge $u_-v_{-2}$, then $AC_1D^-$ is regular
    and $\prob{AC_1D^-}\geq 0.5/256$, which would make the total
    probability exceed $88/256$. Let us therefore assume the contrary.

    The event $C_1AD^-$ is covered by $\pair{\mate{u}{-}}{v_-}2$ and
    $q(C_1AD^-) = 0$, so the probability of $C_1AD^-$ is at least
    $3/4\cdot 0.25/256$. Similarly, $C_1AD^+$ is covered by
    $(v_+,\mate{u}{-})$. Suppose for a moment that this pair is 2-free; we
    then get $\prob{C_1AD^+}\geq 3/4\cdot 0.25/256$. The event
    $C_1AD^0$ is covered by $(v,\mate{u}{-})$ and $(\mate{u}{-},v)$. Our assumptions
    imply for each of the pairs that it is 2-free. By
    Lemma~\ref{l:dep}, $\prob{C_1AD^0} \geq 1/2\cdot 0.25/256$. The
    contribution we have obtained from $C_1AD^+\cup C_1AD^-\cup
    C_1AD^0$ is at least $0.5/256$, which is sufficient to complete
    the proof in this subcase.
    
    \begin{figure}
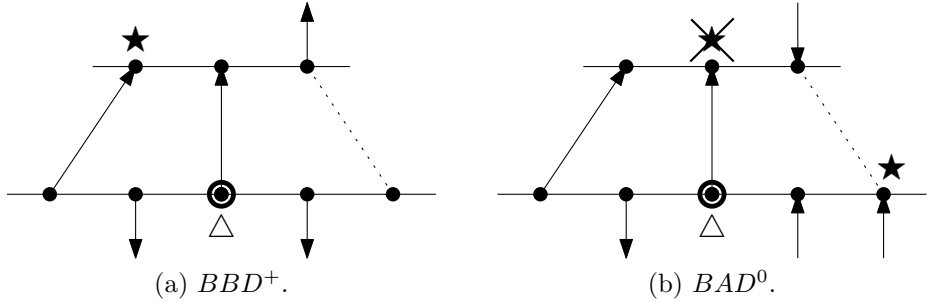

      \centering
      \hf\sfig{69}{$BBD^+$.}\hf\sfig{70}{$BAD^0$.}\hf
      \caption{Some of the events used in Case~\ref{case:nochord-ll} of
        the proof of Lemma~\ref{l:nochord}.}
      \label{fig:nochord-ll}
    \end{figure}

    It remains to consider the possibility that $(v_+,\mate{u}{-})$ is not
    2-free in the diagram for $C_1AD^+$. It must be that the path
    $v_+C_v\mate{u}{-}$ includes $\mate{v}{-}$ and has length 3. The probability
    bound for $C_1AD^+$ is now reduced to $1/2\cdot
    0.25/256$. However, now, $C_1AD^0$ is covered by $(\mate{u}{-},v)$, and
    we find that $\prob{C_1AD^0} \geq 3/4\cdot 0.25/256$. In other
    words,
    \begin{equation*}
      \prob{C_1AD^+\cup C_1AD^-\cup C_1AD^0} \geq 0.5/256
    \end{equation*}
    as before.
  \end{xcase}
  
  By symmetry, it remains to consider the following case. Note that
  our assumption that the set $\Setx{u_-,u_+,v_-,v_+}$ is independent
  remains in effect.

  \begin{xcase}{$G$ contains no edge from the set
      $\Setx{u_{-2},u_{+2}}$ to $\Setx{v_-,v_+}$.}%
    Consider the weakly regular event $BBD^+$
    (Figure~\ref{fig:nochord-final}). As before, the fact that
    $\size{V(Z)}\geq 5$ if $Z$ is odd, together with
    Observation~\ref{obs:q}(ii), implies that $q(BBD^+) \leq
    1/4$. Since the event has weight 7, $\prob{BBD^+} \geq 1.9/256$ by
    Lemma~\ref{l:dep}. We get the same estimate for $BBD^-$ and
    $BBD^0$.

    \begin{figure}
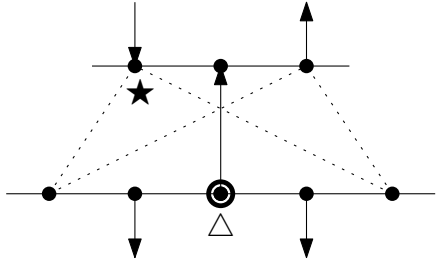

      \centering \hf\fig{71}\hf
      \caption{The event $BBD^+$ used in the final part of the proof of
        Lemma~\ref{l:nochord}.}
      \label{fig:nochord-final}
    \end{figure}

    Since $u_-$ is not adjacent to either of $v_-$ and $v_+$, the event
    $ABD^+$ is valid. It is regular, so $\prob{ABD^+} \geq 0.5/256$. The
    same applies to the events $ABD^-$, $ABD^0$, $BAD^+$, $BAD^-$ and
    $BAD^0$. Thus, the probability of the union of these six events is
    at least $3/256$. Together with the other events described so far,
    the probability is at least $87.9/256$. As in the previous cases,
    this means that we may assume that the mate of $u_+$ is not
    contained in $Z$, for otherwise we would obtain a further $0.8/256$
    from the event $E^\pm$ and reach the required amount.

    Since the length of $C_v$ is at least 5, $\mate{u}{+}$ is not adjacent to
    both $v_-$ and $v_+$. Suppose that it is not adjacent to $v_+$ (the
    other case is symmetric). Then $AC_1D^+$ is covered by the pair
    $\pair{v_+}{\mate{u}{+}}2$. Hence, $\prob{AC_1D^+} \geq 3/4\cdot 0.5/256 =
    0.375/256$. The total probability of $u\in I$ is therefore larger
    than $88/256$, which concludes the proof.
  \end{xcase}
\end{proof}


\section{Analysis: $uv$ is a chord}
\label{sec:chord}

In the present section, we continue the analysis of
Section~\ref{sec:nochord}, this time confining our attention to the
case where $uv$ is a chord of $Z$. Although this case is more
complicated, one useful simplification is that by
Observation~\ref{obs:q}(i), we now have $q(\Delta) = 0$ for any
template $\Delta$. In particular, $\prob{E^\pm} \geq 16/256$, which
implies
\begin{equation*}
  \prob{E^0 \cup E^- \cup E^+ \cup E^\pm} \geq \frac{80}{256}.
\end{equation*}
Roughly speaking, since the probability needed to prove
Theorem~\ref{t:main} is $88/256$, we need to find events in $\Sigma$
whose total probability is at least $8/256$. However, like in
Section~\ref{sec:nochord}, we may actually require a higher
probability or be satisfied with a lower one, depending on the type of
the vertex. The surplus probability will be used to compensate for the
deficits in Section~\ref{sec:phase5}.

\begin{table}
  \centering
  \begin{tabular}{c|m{.7\textwidth}|c}
    type of $u$ & \centering condition & $\eps(u)$\\\hline
    I & the path $v_-vv_+$ is contained in a
    4-cycle in $G$, neither the path $u_-uu_+$ nor the edge $uv$ are
    contained in a 4-cycle, and $u$ is not of types Ia, Ib, Ia$^*$ or Ib$^*$
    (see text) & $-0.5$ \\\hline
    Ia & $\size{uZv} = 4$ and $M$ includes
    the edges $u_{+2}v_+$, $u_{-2}v_-$, 
    while $u_+v_{+2}\notin E(M)$ & $-2$ \\\hline
    Ib & $\size{uZv} = 4$ and $M$ includes
    the edges $u_{+2}v_+$, $u_{-2}v_-,u_+v_{+2}$, & $-1.5$ \\\hline
    II & $\size{uZv} = 4$, $\size{vZu} \geq 7$ and $M$ includes all of the edges
    $u_-v_{+2}$, $u_{-2}v_+$, $u_{-3}u_+$, 
    while $v_{+3}v_-\notin E(M)$ & $-0.125$ \\\hline
    IIa & $\size{uZv} = 4$, $\size{vZu} = 6$, and $M$ includes all of the
    edges $u_{-2}v_+$, $u_{-3}u_+$ and $u_-u_{-4}$, & $-0.5$ \\\hline
    III & $\size{uZv} = 4$, $\size{vZu} = 8$ and $M$ includes all
    of the edges $u_{-2}v_+$, $u_{-3}u_+$, $v_{+3}v_-$ and 
    $u_-u_{-4}$ & $-0.125$
  \end{tabular}
  \caption{The type of a deficient vertex $u$ provided that $uv$ is a
    chord of $Z$, and the associated value $\eps(u)$.}
  \label{tab:chord-type}
\end{table}

Recall that at the beginning of Section~\ref{sec:nochord}, we defined
deficient vertices of type 0, and we associated a number $\eps(u)$
with the vertex $u$ provided that $uv$ is not a chord of a cycle of
$F$. We are now going to provide similar definitions for the opposite
case, introducing a number of new types of deficient vertices.

Suppose that $uv$ is a chord of $Z$ which is not contained in any
4-cycle of $G$. The vertex $u$ is \emph{deficient} if it satisfies one
of the conditions in Table~\ref{tab:chord-type}. (See the
illustrations in Figure~\ref{fig:deficient}.) Since the conditions are
mutually exclusive, this also determines the \emph{type} of the
deficient vertex $u$.

We now extend the definition to cover the symmetric
situations. Suppose that $u$ satisfies the condition of type II when
the implicit orientation of $Z$ is replaced by its reverse --- which
also affects notation such as $u_+$, $uZv$ etc. In this case, we say
that $u$ is deficient of type II$^*$. (As seen in
Figure~\ref{fig:symmetric}, the picture representing the type is
obtained by a flip about the vertical axis.) The same notation is used
for all the other types except types 0 and I. A type such as II$^*$ is
called the \emph{mirror} type of type II.

Note that even with this extension, the types of a deficient vertex
remain mutually exclusive. Furthermore, we have the following
observation which will be used repeatedly without explicit mention:
\begin{observation}
  If $u$ is deficient (of type different from 0), then its mate $v$ is
  not deficient.
\end{observation}
\begin{proof}
  Let $u$ be as stated. A careful inspection of
  Table~\ref{tab:chord-type} and Figure~\ref{fig:deficient} shows that
  the path $u_-uu_+$ is not contained in any 4-cycle. It follows that
  $v$ is not deficient of type I, Ia, Ib or their mirror
  variants. Suppose that $v$ is deficient. By symmetry, $u$ also does
  not belong to the said types, and hence the types of both $u$ and
  $v$ are II, IIa, III or the mirror variants. As seen from
  Figure~\ref{fig:deficient}, when $u$ is of any of these types, the
  path $u_-uu_+$ belongs to a 5-cycle in $G$. By symmetry again, the
  same holds for $v_-vv_+$. The only option is that $u$ belongs to
  type III and $v$ to III$^*$, or vice versa. But this is clearly
  impossible: if $u$ is of type III or III$^*$, then one of its
  neighbours on $Z$ is contained in a 4-cycle, and this is not the
  case for any neighbour of $v$ on $Z$. Hence, $v$ cannot be of type
  III or III$^*$. This contradiction shows that $v$ is not deficient.
\end{proof}

We will often need to apply the concept of a type to the vertex $v$
rather than $u$. This may at first be somewhat tricky; for instance,
to obtain the definition of `$v$ is of type IIa$^*$', one needs to
interchange $u$ and $v$ in the definition of type IIa in
Table~\ref{tab:chord-type} and then perform the reversal of the
orientation of $Z$. In this case, the resulting condition will be that
$\size{uZv} = 4$, $\size{vZu} = 6$ (here the two changes cancel each
other) and $M$ includes the edges $v_{+2}u_-$, $v_{+3}v_-$ and
$v_+v_{+4}$. To spare the reader from having to turn
Figure~\ref{fig:deficient} around repeatedly, we picture the various
cases where $v$ is deficient in Figure~\ref{fig:v-deficient}.

Table~\ref{tab:chord-type} also associates the value $\eps(u)$ with
each type. By definition, a type with an asterisk (such as II$^*$) has
the same value assigned as the corresponding type without an asterisk.

We now extend the function $\eps$ to all vertices of $G$. It has been
defined for all deficient vertices, as well as for all vertices whose
mate is contained in a different cycle of $F$. Suppose that $w$ is a
non-deficient vertex whose mate $w'$ is contained in the same cycle of
$F$. We set
\begin{equation*}
  \eps(w) =
  \begin{cases}
    -\eps(w') & \text{ if $w'$ is deficient,}\\
    0 & \text{ otherwise.}
  \end{cases}
\end{equation*}

\begin{figure}
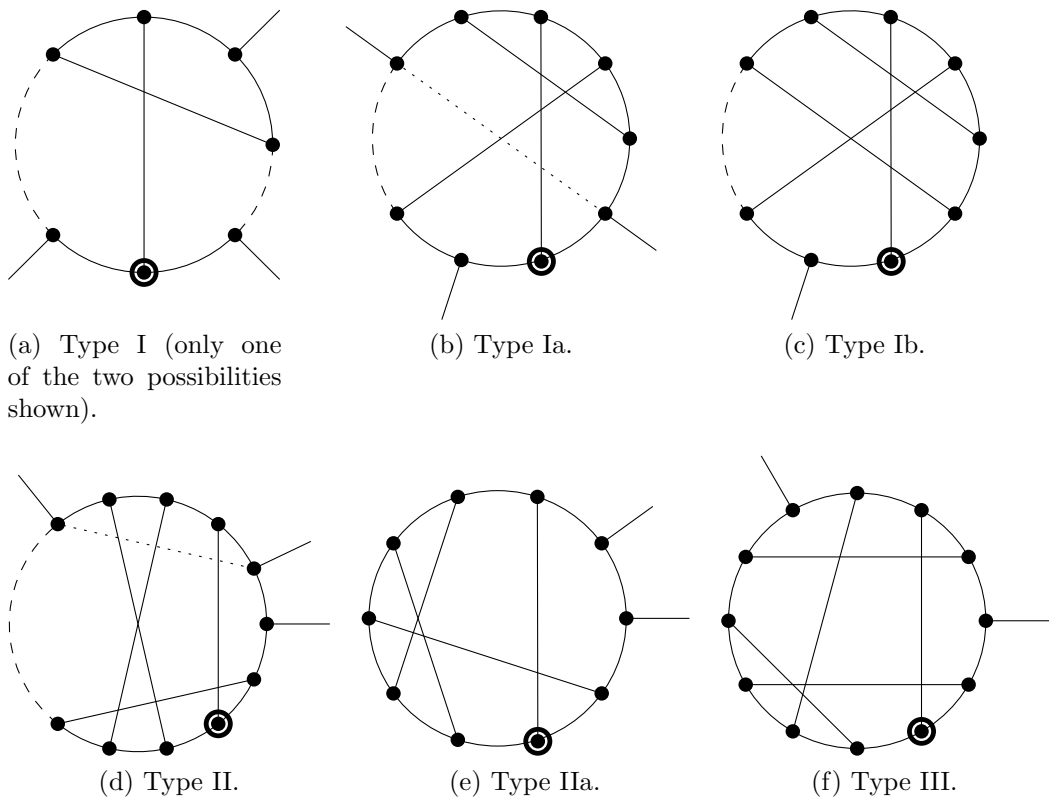

  \centering%
  \sfigdef{51}%
  \sfigtop{38}{Type I (only one of the two possibilities shown).}\hf
  \sfigtop{50}{Type Ia.}\hf\sfigtop{51}{Type Ib.}\hspace{8mm}\hspace*{0mm}\\
  \sfigdef{35}%
  \sfig{37}{Type II.}\hf\sfig{36}{Type IIa.}\hf
  \sfig{35}{Type III.}
  \caption{Deficient vertices.}
  \label{fig:deficient}
\end{figure}

\begin{figure}
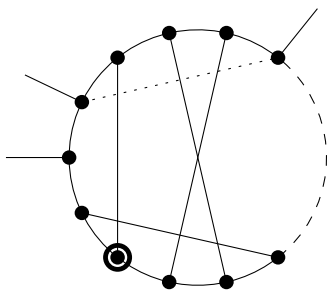

  \centering
  \hf\fig{72}\hf
  \caption{A deficient vertex $u$ of type II$^*$.}
  \label{fig:symmetric}
\end{figure}

\begin{figure}
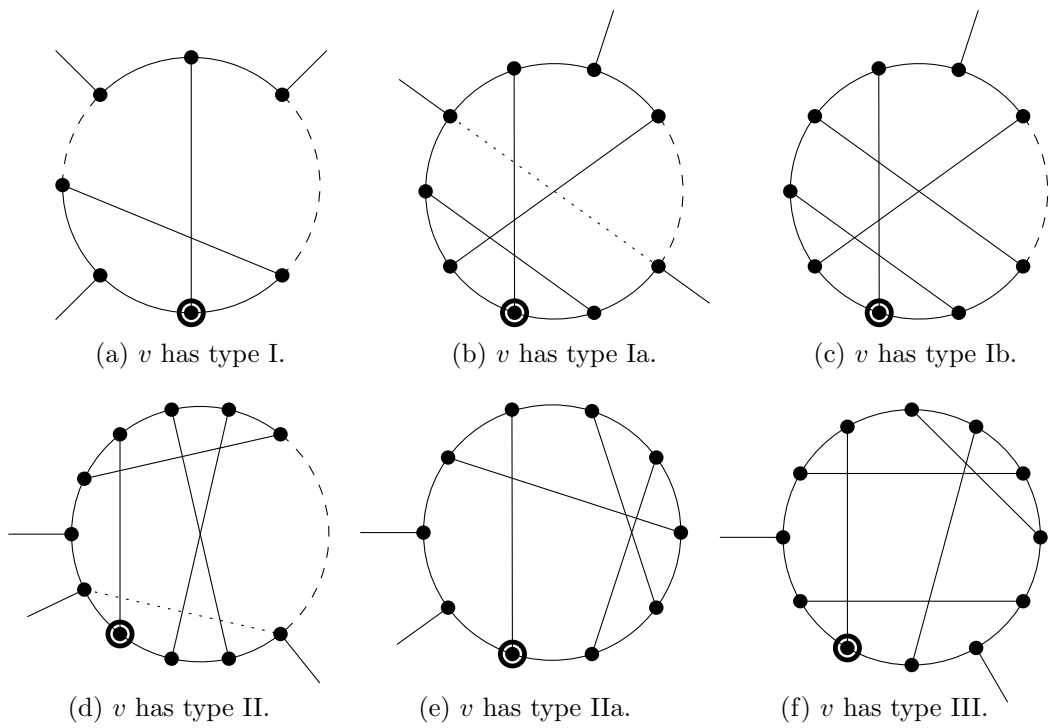

  \centering%
  \hspace{6mm}\sfigdef{90}%
  \sfig{88}{$v$ has type I.}\hf
  \sfig{89}{$v$ has type Ia.}\hf\sfigtop{90}{$v$ has type Ib.}\\
  \sfigdef{91}%
  \sfigtop{91}{$v$ has type II.}\hf\sfigtop{92}{$v$ has type IIa.}\hf
  \sfigtop{93}{$v$ has type III.}
  \caption{The situation when the vertex $v$ is deficient. As usual,
    the vertex $u$ is circled.}
  \label{fig:v-deficient}
\end{figure}

Our goal in this section is to prove the following proposition, which
is the main technical result of this paper. As in the proof of
Lemma~\ref{l:nochord}, we mark the end of each case by $\blacktriangle$;
furthermore, the end of each subcase is marked by $\triangle$.

\begin{proposition}\label{p:32}
  If $uv$ is a chord of $Z$, then for the total probability of the
  events in $\Sigma$ we have
  \begin{equation*}
    \prob{\bigcup\Sigma} \geq \frac{8+\eps(u)}{256}.
  \end{equation*}
\end{proposition}

\begin{proof}
We distinguish a number of cases based on the structure of the
neighbourhood of $u$ in $G$.

\begin{figure}
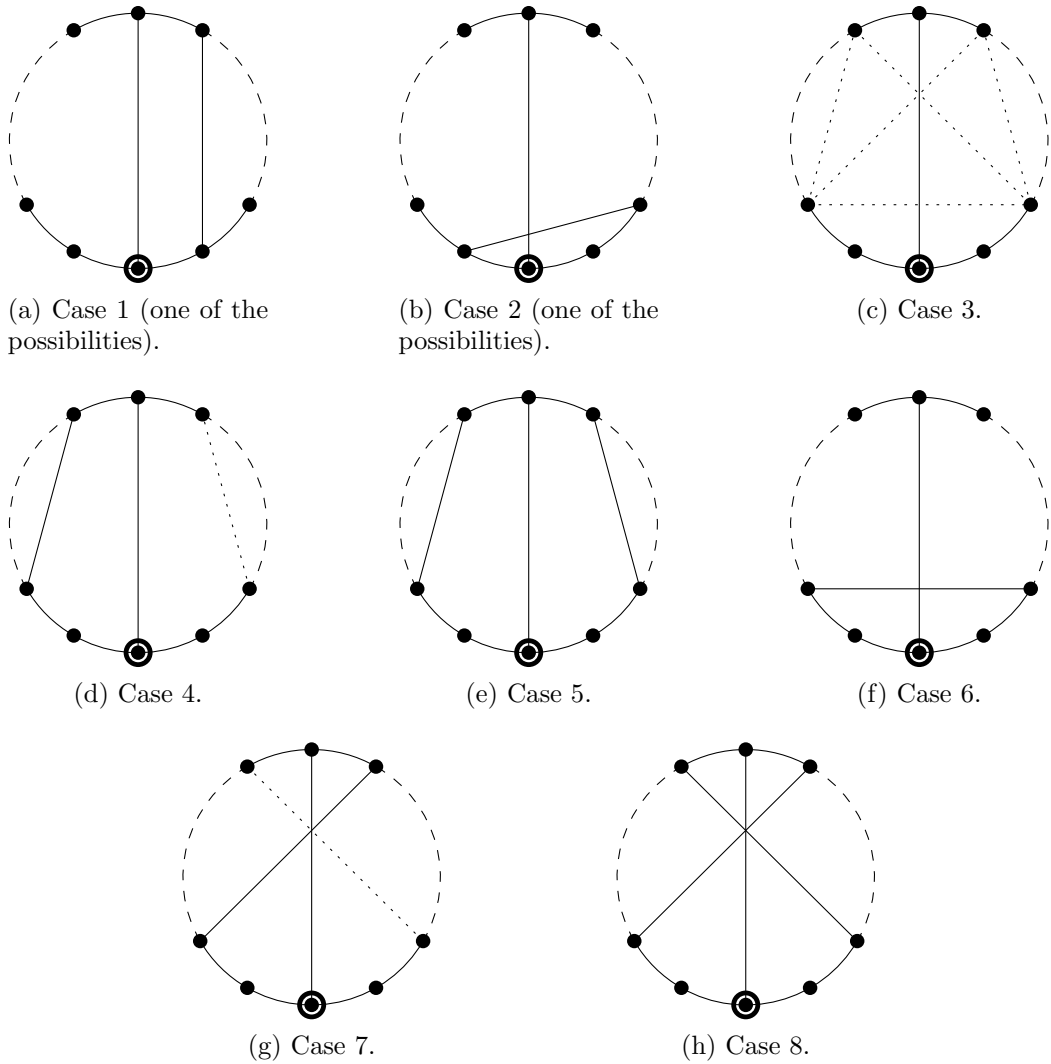

  \centering
  \sfig{41}{Case 1 (one of the possibilities).}
  \hf\sfig{42}{Case 2 (one of the possibilities).}
  \hf\sfig{43}{Case 3.}\\
  \sfig{44}{Case 4.}
  \hf\sfig{45}{Case 5.}
  \hf\sfig{46}{Case 6.}\\
  \hf\sfig{47}{Case 7.}
  \hf\sfig{48}{Case 8.}\hf
  \caption{The main cases in the proof of
    Proposition~\ref{p:32}. Relevant non-edges are represented by
    dotted lines, paths are shown as dashed lines.}
  \label{fig:cases}
\end{figure}

\setcounter{xcasehdr}0  
 
\begin{xcase}{The edge $uv$ is contained in a
    4-cycle.}\label{case:c4}%
    Observe that in this case, neither $u$ nor $v$ is deficient.

    Suppose that $uvv_-u_+$ is a 4-cycle (the argument in the other
    cases is the same). Consider first the possibility that $v_-u_+$
    is an edge of $M$. The event $BBD^0$ is (valid and) regular. By
    Lemma~\ref{l:dep}, $\prob{BBD^0} \geq 4/256$. Since this lower
    bound increases to $8/256$ if $u_-v_+$ is an edge of $M$ (and
    since $v$ is not deficient), we may actually assume that this is
    not the case. Consequently, $\prob{BBD^-} \geq 4/256$ as $BBD^-$
    is regular. The total contribution is $8/256$ as desired.

    We may thus assume that $v_-u_+$ is an edge of $F$ and no edge of
    $M$ has both endvertices in $\Setx{u_-,u_+,v_-,v_+}$. Since the
    events $BBD^0$ and $BBD^-$ are regular, we have $\prob{BBD^0\cup
      BBD^-} \geq 4/256$. 

    A further probability of $4/256$ is provided by the regular events
    $BAD^0$ and $BAD^-$. Indeed, although the template $BAD^0$ has
    weight 8, which would only yield $\prob{BAD^0}\geq 1/256$ by
    Lemma~\ref{l:dep}, the estimate is improved to $2/256$ by the fact
    that the associated diagram has a removable symbol at $v$. The
    same applies to the event $BAD^-$. We conclude
    \begin{equation*}
      \prob{BBD^0\cup BBD^-\cup BAD^0\cup BAD^-} \geq 8/256
    \end{equation*}
    as required.
  \end{xcase}

  We will henceforth assume that $uv$ is not contained in a
  4-cycle. Note that this means that the set $\Setx{u_-,u_+,v_-,v_+}$
  is independent. Consider the regular event $BBD^+$
  (Figure~\ref{fig:bb}). By Lemma~\ref{l:dep}, we have
  \begin{equation*}
    \prob{BBD^+} \geq \frac2{256}.
  \end{equation*}
  The same applies to the events $BBD^0$ and $BBD^-$. Thus, in the
  subsequent cases, it suffices to find additional events of total
  probability at least $(2+\eps(u))/256$.

  \begin{figure}
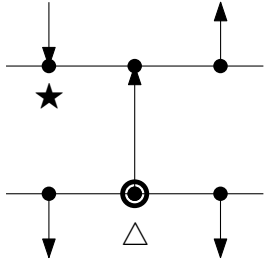

    \centering
    \fig{18}
    \caption{The event $BBD^+$.}
    \label{fig:bb}
  \end{figure}
  

  \begin{xcase}{The path $u_-uu_+$ is contained in a 4-cycle.}%
    \label{case:c4-down}%
    Suppose that $u_-uu_+u_{+2}$ is such a 4-cycle. (The other case is
    symmetric.) Consider the events $C_1AD^+$ and $BAD^+$. Since the
    condition of Case~\ref{case:c4} does not hold, and by the
    assumption that $G$ is triangle-free, the set $\Setx{u_+,v_-,v_+}$
    is independent in $G$. Furthermore, each of the events is regular
    and by Lemma~\ref{l:dep}, each of them has probability at least
    $1/256$. Thus, it remains to find an additional contribution of
    $\eps(u)$.

    We distinguish several subcases based on the deficiency and type
    of the vertex $v$. Since $u_-uu_+$ is contained in a 4-cycle, $v$
    is either not deficient, or is deficient of type I, Ia, Ib, Ia$^*$
    or Ib$^*$.

    \begin{xxcase}{$v$ is not deficient.}%
      In this subcase, $\eps(u) \leq 0$, so there is nothing to prove.
    \end{xxcase}

    \begin{xxcase}{$v$ is deficient of type I.}\label{subc:c4-down-1}%
      By the definition of type I, both of the following conditions
      hold:
      \begin{itemize}
      \item $u_+v_{+2} \notin E(M)$ or $\size{uZv} \geq 5$,
      \item $u_-v_{-2} \notin E(M)$ or $\size{vZu} \geq 5$.
      \end{itemize}
      Moreover, we have $\eps(u)=0.5$.

      We may assume that $M$ includes the edge $u_{-2}v_-$, for
      otherwise the event $ABD^-$ is regular (see
      Figure~\ref{fig:c4-down-1}(a)) and has probability at least
      $0.5/256$ as required.

      The event $ABD^+$ (Figure~\ref{fig:c4-down-1}(b)) is covered by
      the pair $(v_+,u_{-2})$. Consequently, we may assume that
      $\size{vZu} = 4$: otherwise the pair is 1-free, and since the
      event has weight 8, we have $\prob{ABD^+}\geq 0.5/256$ by
      Lemma~\ref{l:dep}.

      By a similar argument applied to the event $C_1AD^-$, we infer
      that $\size{uZv}=4$. Thus, the length of $Z$ is 8 and the
      structure of $G[V(Z)]$ is as shown in
      Figure~\ref{fig:c4-down-2}(a). The regular event $C_1C_2D^+$
      (Figure~\ref{fig:c4-down-2}(b)) has probability at least
      $0.5/256$, which is sufficient. This concludes the present
      subcase.

      \begin{figure}
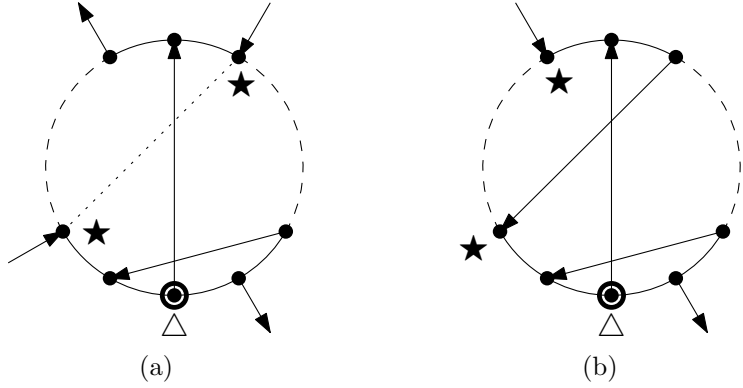

        \centering
        \hf\sfig{78}{}\hf\sfig{77}{}\hf
        \caption{Subcase~\ref{subc:c4-down-1} of the proof of
          Proposition~\ref{p:32}: (a) The event $ABD^-$ if
          $u_{-2}v_-\notin E(M)$. (b) The event $ABD^+$.}
        \label{fig:c4-down-1}
      \end{figure}

      \begin{figure}
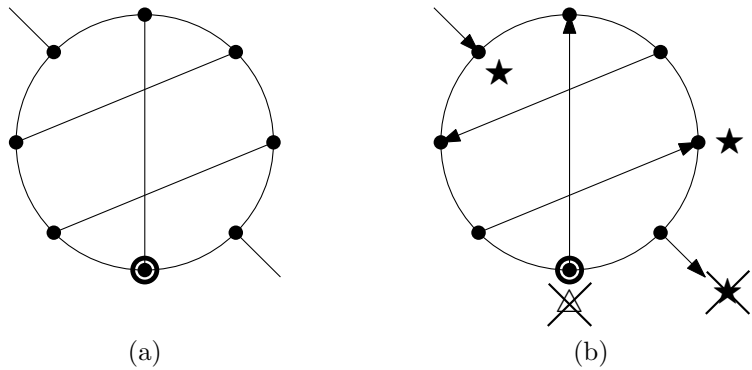

        \centering
        \sfigdef{94}
        \hf\sfigtop{52}{}\hf\sfigtop{94}{}\hf
        \caption{(a) A configuration in Subcase~\ref{subc:c4-down-1}
          of the proof of Proposition~\ref{p:32}. (b) The event
          $C_1C_2D^+$.}
        \label{fig:c4-down-2}
      \end{figure}
    \end{xxcase}

    \begin{xxcase}{$v$ is deficient of type Ia, Ib, Ia$^*$ or
        Ib$^*$.}%
      By symmetry, we may assume that $v$ is either of type Ia$^*$ (if
      $u_{-2}v_-$ is not an edge of $M$) or Ib$^*$
      (otherwise). Accordingly, we have either $\eps(u) = 2$ or
      $\eps(u) = 1.5$.

      The regular event $C_1C_2D^+$ provides a contribution of
      $1/256$. If $u_{-2}v_- \notin E(M)$ (thus, $v$ is of type Ia$^*$
      and $\eps(u) = 2$), then the event $ABD^-$ is also regular
      (including when $\size{vZu}=4$) and $\prob{ABD^-} \geq 1/256$, a
      sufficient amount.

      It remains to consider the case that $u_{-2}v_-\in E(M)$. The
      required additional probability of $0.5/256$ is supplied by the
      event $ABD^+$, which is covered by the 1-free pair
      $(v_+,u_{-2})$. 
    \end{xxcase}

    The discussion of Case~\ref{case:c4-down} is complete.
  \end{xcase}
  
  From here on, we assume that none of the conditions of Cases 1 and 2
  holds. In particular, $v$ is not deficient of type I, Ia, Ib or
  their mirror types. We distinguish further cases based on the set of
  edges induced by $M$ on the set
  \begin{equation*}
    U = \Setx{u_{-2},u_{+2},v_-,v_+}.
  \end{equation*}
  Note that the length of the paths $uZv$ and $vZu$ is now assumed to
  be at least 4. We call a path \emph{short} if its length equals 4.


  \begin{xcase}{$E(M[U]) = \emptyset$.}\label{case:empty}%
    We claim that if $v$ is deficient, then its type is III or
    III$^*$. Indeed, for types I, Ia, Ib and their mirror types,
    $u_-uu_+$ would be contained in a 4-cycle and this configuration
    has been covered by Case~\ref{case:c4-down}. For types II, IIa and
    their mirror variants, $U$ would not be an independent set. Since
    type 0 is ruled out for trivial reasons, types III and III$^*$ are
    the only ones that remain. The only subcase compatible with these
    types is Subcase~\ref{subc:0-leftshort}; in the other subcases,
    $v$ is not deficient and we have $\eps(u) \leq 0$. This will
    simplify the discussion in the present case.
    
    We begin by considering the event $ABD^-$. By the assumptions, it
    is valid. Since neither $(u_{-2},v_-)$ nor its reverse is a
    sensitive pair, the event is regular. Thus, $\prob{ABD^-} \geq
    0.5/256$. By symmetry, we have $\prob{BAD^+} \geq 0.5/256$.

    We distinguish several subcases, in each of which we try to
    accumulate further $(1+\eps(u))/256$ worth of probability.


    \begin{xxcase}{None of $uZv$ and $vZu$ is short.}%
      \label{subc:0-noshort}%
      Consider the event $ABD^0$. By the assumptions, it is valid and
      covered by $(v_+, u_{-2})$. Since $vZu$ is not short and the
      diagram of $ABD^0$ contains only one outgoing arc (namely
      $u_+\mate{u}{+}$), the pair is 1-free. By Lemma~\ref{l:dep},
      $\prob{ABD^0} \geq 1/2 \cdot 0.5/256 = 0.25/256$. By symmetry,
      $\prob{BAD^0} \geq 0.25/256$.

      The argument for $ABD^0$ also applies to the event $ABD^+$
      (whose diagram has two outgoing arcs), unless the vertex set of
      the path $vZu$ is $\Setx{v,v_+,\mate{u}{+},\mate{v}{-},u_-,u}$ (in which case
      we get the two possibilities in Figure~\ref{fig:indep-abd0}). If
      this does not happen, then we obtain a contribution of at least
      $0.25/256$ again.

      \begin{figure}
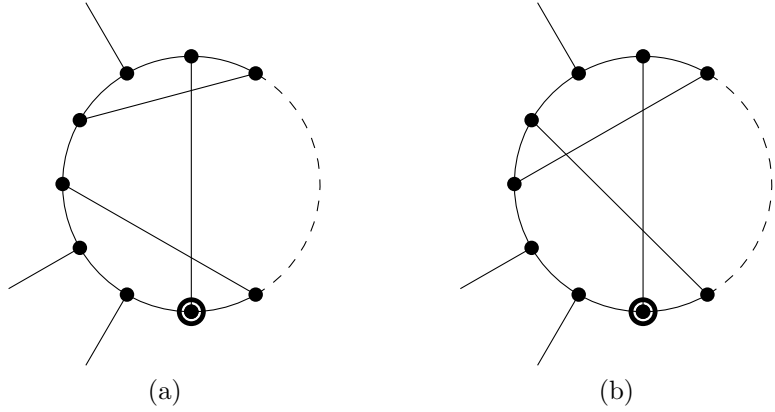
 \centering 
        \hf\sfig{22}{}\hf 
        \sfig{23}{}\hf
        \caption{Two cases where the event $ABD^+$ cannot be used in
          Subcase~\ref{subc:0-noshort} of the proof of
          Proposition~\ref{p:32}.}
        \label{fig:indep-abd0}
      \end{figure}

      Let us examine the exceptional case in
      Figure~\ref{fig:indep-abd0}(a) (i.e., $\mate{u}{+} = u_{-3}$ and $\mate{v}{-}
      = v_{+2}$). The event $C_1C_1D^+$ is covered by
      $\pair{\mate{u}{-}}{\mate{u}{-}}4$. By Lemma~\ref{l:dep}, $\prob{C_1C_1D^+}
      \geq 79/80 \cdot 1/256 > 0.98/256$.

      Consider now the situation of
      Figure~\ref{fig:indep-abd0}(b). The event $AAD^+$ is valid,
      since $\Setx{u_{-2},u_{+2},v_+}$ is an independent set by
      assumption, and it is regular. We infer that $\prob{AAD^+} \geq
      0.25/256$.
      
      To summarize the above three paragraphs, we proved
      \begin{equation*}
        \prob{ABD^+ \cup C_1C_1D^+ \cup AAD^+} \geq 0.25/256.
      \end{equation*}
      By symmetry, we have
      \begin{equation*}
        \prob{BAD^- \cup C_1C_1D^- \cup AAD^-} \geq 0.25/256.
      \end{equation*}
      Together with the events $ABD^0$ and $BAD^0$ considered earlier,
      this makes for a total contribution of at least $1/256$. As
      noted at the beginning of Case~\ref{case:empty}, $\eps(u) \leq
      0$, so this is sufficient.
    \end{xxcase}


    \begin{xxcase}{The path $vZu$ is short, but $uZv$ is
        not.}\label{subc:0-leftshort}%

      In this subcase, $v$ may be deficient of type III$^*$, in
      which case $\eps(u) = 0.125$; otherwise, $\eps(u)\leq 0$.

      The event $BAD^-$ is covered by the pair $(u_{+2},v_-)$ which is
      1-free unless $\mate{v}{+}$ and $\mate{u}{-}$ are the only internal vertices
      of the path $u_{+2}Zv_-$. However, this situation would be
      inconsistent with our choice of $F$, since $\bd Z$ would have
      size 4. (Recall that $\bd Z$ is the set of edges of $G$ with one
      end in $V(Z)$.) Consequently, $\prob{BAD^-} \geq 1/2 \cdot
      0.5/256 = 0.25/256$. Moreover, if $\mate{u}{+}$ (which is a tail in
      $BAD^-$) is contained in $u_{+2}Zv_-$, then $\prob{BAD^-}\geq
      0.5/256$.

      The same discussion applies to the event $BAD^0$. In particular,
      if $\mate{u}{+} \in V(u_{+2}Zv_-)$, then the probability of the union
      of these two types is at least $1/256$. This is a sufficient
      amount, unless $v$ is deficient of type III$^*$, in which case
      a further $0.25/256$ is obtained from the regular event $AAD^-$.

      We may thus assume that $\mate{u}{+} \notin V(u_{+2}Zv_-)$ (so $v$ is
      not deficient). The event $AC_1D^-$ is then covered by
      $\pair{\mate{u}{+}}{\mate{u}{+}}3$ (we are taking into account the arc
      incident with $v_+$) and hence $\prob{AC_1D^-} \geq 39/40 \cdot
      0.25/256 > 0.24/256$ by Lemma~\ref{l:dep}.

      The event $AC_2D^-$ is covered by the pair
      $\pair{u_{+2}}{v_-}{1}$ and has probability at least $1/2 \cdot
      0.0625/256 > 0.03/256$. We claim that $\prob{BAD^-\cup BAD^0\cup
        AAD^-} \geq 0.75/256$. Since the total amount will exceed
      $1/256$, this will complete the present subcase.

      Suppose first that $\mate{v}{+}\in V(uZv)$. Then the event $BAD^0$ is
      regular and $\prob{BAD^0} \geq 0.5/256$. In addition, $BAD^-$
      has only one sensitive pair $(u_{+2},v_-)$. This pair is 1-free,
      for otherwise $\mate{v}{+}$ and $\mate{u}{-}$ would be the only internal
      vertices of the path $u_{+2}Zv_-$, and $Z$ would be incident
      with exactly four non-chord edges of $M$, a contradiction with
      the choice of $F$. Thus, $\prob{BAD^-} \geq 0.25/256$ and the
      claim is proved.

      Let us therefore assume that $\mate{v}{+}\notin V(uZv)$. We again
      distinguish two possibilities according to whether $\mate{u}{-}$ is
      contained in $uZv$ or not. If $\mate{u}{-}\in V(uZv)$, then
      $\prob{BAD^-} \geq 1/2 \cdot 0.5/256 = 0.25/256$ as $BAD^-$ is
      covered by $\pair{u_{+2}}{v_-}1$. Similarly, $\prob{BAD^0} \geq
      0.25/256$. The event $AAD^-$ is regular of weight 10, whence
      $\prob{AAD^-} \geq 0.25/256$. The total probability of these
      three events is at least $0.75/256$ as claimed.

      To complete the proof of the claim, we may assume that
      $\mate{u}{-}\notin V(uZv)$. The only possibly sensitive pair of $BAD^-$
      and $BAD^0$ is now 2-free, implying a probability bound of $3/4
      \cdot 0.5/256$ for each event. Thus, $\prob{BAD^-\cup BAD^0}
      \geq 0.75/256$, finishing the proof of the claim and the whole
      subcase. 
    \end{xxcase}
      

    \begin{xxcase}{Both $vZu$ and $uZv$ are short.}%
      In this subcase, $Z$ is an 8-cycle; by our assumptions, it has
      only one chord $uv$. Recall also that in this subcase, $\eps(u)
      \leq 0$.
        
      Consider the event $AC_1D^-$. Since it is covered by
      $\pair{\mate{u}{+}}{\mate{u}{+}}3$, we have $\prob{AC_1D^-} \geq 39/40 \cdot
      0.25/256 > 0.24/256$ by Lemma~\ref{l:dep}. By symmetry,
      $\prob{C_1AD^+} \geq 0.24/256$, so the total probability so far
      is $0.48/256$.

      Suppose now that the vertices $\mate{u}{+}$ and $\mate{u}{-}$ are located on
      different cycles of $F$. By Lemma~\ref{l:dep}, $\prob{C_1C_1D^0}
      \geq 39/40\cdot 0.5/256 > 0.48/256$. Similarly,
      $\prob{C_1C_1D^+} \geq 77/80 \cdot 0.5/256 > 0.48/256$, which
      makes for a sufficient contribution.

      We may thus assume that $\mate{u}{+}$ and $\mate{u}{-}$ are on the same cycle,
      say $Z'$, of $F$. Suppose that they are non-adjacent, in which
      case $C_1C_1D^0$ is covered by $\pair{\mate{u}{+}}{\mate{u}{-}}2$ and
      $\pair{\mate{u}{-}}{\mate{u}{+}}2$, and its probability is at least $1/2 \cdot
      0.5/256 = 0.25/256$. If neither $\mate{v}{-}$ nor $\mate{v}{+}$ are on $Z'$,
      then the same computation applies to $C_1C_1D^+$ and
      $C_1C_1D^-$, so the total probability accumulated so far is
      $(0.48 + 0.25 + 0.25 + 0.25)/256 > 1/256$ by
      Lemma~\ref{l:dep}. We may thus assume, without loss of
      generality, that $\mate{v}{+} \in V(\mate{u}{+}Z'\mate{u}{-})$. Under this
      assumption, $C_1C_1D^-$ is covered by $\pair{\mate{u}{-}}{\mate{u}{+}}1$ and
      thus $\prob{C_1C_1D^-} \geq 1/2 \cdot 0.5/256 = 0.25/256$. At
      the same time, $\prob{C_1C_1D^+}$ is similarly seen to be at
      least $0.125/256$, which makes the total probability at least
      $(0.48 + 0.25 + 0.25 + 0.125)/256 > 1/256$.

      It remains to consider the possibility that $\mate{u}{+}$ and $\mate{u}{-}$
      are adjacent. In this case, $\prob{AC_1D^- \cup C_1AD^+} \geq
      0.5/256$, so we need to find additional $0.5/256$. The event
      $C_1C_2D^+$ has a template covered by $\pair{\mate{u}{-}}{\mate{u}{-}}2$, and
      hence its probability is at least $19/20 \cdot 0.25/256 >
      0.23/256$. Similarly, $\prob{C_2C_1D^-} \geq 0.23/256$. The same
      argument applies to the events $C_1C_3D^+$ and $C_3C_1D^-$,
      resulting in a total probability of $(0.5 + 4\cdot 0.23)/256 >
      1/256$. This finishes Case 3.
    \end{xxcase}
  \end{xcase}
  

  \begin{xcase}{$E(M[U]) = \Setx{u_{-2}v_+}$.}%
    \label{case:ll}%
    In this case, two significant contributions are from the regular
    events $ABD^-$ and $BAD^+$:
    \begin{align*}
      \prob{ABD^-} &\geq \frac1{256},\\
      \prob{BAD^+} &\geq \frac{0.5}{256}.
    \end{align*}
    We distinguish several subcases; in each of them, we try to
    accumulate a contribution of $(0.5+\eps(u))/256$ from other
    events.  In particular, if $u$ is deficient of type I, IIa or
    IIa$^*$ (and $\eps(u) = -0.5$), we are done. 

    Let us consider the vertex $v$. We claim that if $v$ is deficient,
    then it must be of type II$^*$ or IIa$^*$. Indeed, the assumption
    that $u_-uu_+$ is not contained in a 4-cycle excludes types I, Ia,
    Ib and their mirror variants. Type 0 is excluded for trivial
    reasons. An inspection of the type definitions shows that if $v$
    is of type II or IIa, then $M$ includes the edge $u_{+2}v_-$,
    which we assume not to be the case. Finally, if $v$ is of type III
    or III$^*$, then $u_{-2}v_+$ is not an edge of $M$, another
    contradiction with our assumption.

    The only types that remain for $v$ are II$^*$ and IIa$^*$. Observe
    that if $v$ is of one of these types, then $uZv$ is short. 
    \begin{xxcase}{The path $uZv$ is not
        short.}\label{subc:ll-notshort}%
      By the above, $v$ is not deficient of either type, whence
      $\eps(u) \leq 0$.  The event $BAD^-$ is covered by
      $\pair{u_{+2}}{v_-}1$ (consider the outgoing arc incident with
      $u_-$). It follows that $\prob{BAD^-} \geq 0.25/256$. The same
      argument applies to $BAD^0$, and thus
      \begin{equation*}
        \prob{ABD^- \cup BAD^+ \cup BAD^- \cup BAD^0} \geq
        \frac{1+0.5+0.25+0.25}{256} = \frac2{256}. 
      \end{equation*}
    \end{xxcase}

    We have observed that if $v$ is deficient, then it must be of type
    II$^*$ or IIa$^*$. Since this requires that the $F$-neighbours of
    $\mate{u}{-}$ are $v_+$ and $\mate{v}{-}$, it can only happen in the following
    subcase.

    \begin{xxcase}{The vertices $\mate{u}{+}$ and $\mate{u}{-}$ are
        non-adjacent.}\label{subc:ll-notadj}%
      Consider the events $AC_1D^-$ and $C_1C_1D^-$
      (Figure~\ref{fig:ll-notadj}). If the event $AC_1D^-$ has a
      sensitive pair, it is either $(\mate{u}{+},\mate{u}{+})$ or $(\mate{u}{+},u_{-2})$.

      \begin{figure}
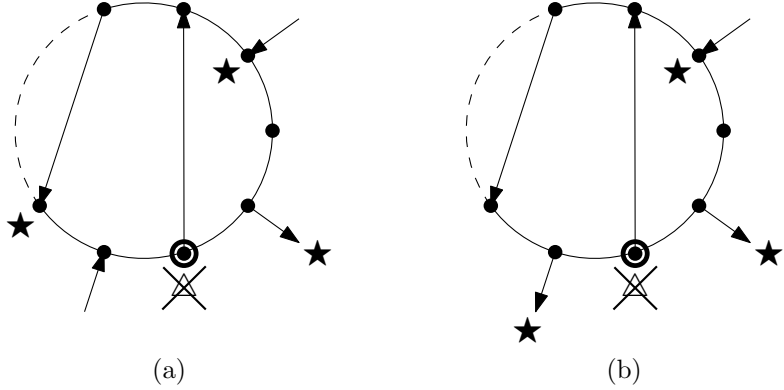

        \centering 
        \sfigdef{97}
        \hf\sfigtop{96}{}\hf\sfigtop{97}{}\hf
        \caption{Two events used in Subcase~\ref{subc:ll-notadj} of
          the proof of Proposition~\ref{p:32}: (a) $AC_1D^-$, (b)
          $C_1C_1D^-$.}
        \label{fig:ll-notadj}
      \end{figure}
      
      Suppose first that $\mate{u}{+}$ is distinct from $u_{-3}$. In this
      case, Lemma~\ref{l:dep} implies that $\prob{AC_1D^-} \geq 3/4
      \cdot 0.5/256$ no matter whether $\mate{u}{+}\in V(Z)$ or
      not. Secondly, $\prob{C_1C_1D^-} \geq 1/2 \cdot 0.5/256$ (by
      Lemma~\ref{l:dep} again), so the total contribution is at least
      $0.625/256$, which is sufficient if $v$ is not deficient, or is
      deficient of type II$^*$. It remains to consider the possibility
      that $v$ is deficient of type IIa$^*$. In this case, $AC_1D^-$ is
      covered by $\pair{\mate{u}{+}}{\mate{u}{+}}4$; by Lemma~\ref{l:dep},
      $\prob{AC_1D^-} \geq 79/80\cdot 0.5/256 > 0.49/256$. Similarly,
      we obtain $\prob{C_1C_1D^-} > 0.49/256$ and $\prob{C_2C_1D^-}
      \geq 0.24/256$. The total contribution is $1.22/256 >
      (0.5+\eps(u))/256$.

      We may thus suppose that $\mate{u}{+} = u_{-3}$; since this is
      incompatible with $v$ being of type II$^*$ as well as IIa$^*$,
      we find that $v$ is not deficient and $\eps(u)\leq 0$. We have
      $\prob{C_1C_1D^-} \geq 3/4 \cdot 0.5/256$ (whether $\mate{u}{-}$ is
      contained in $vZu$ or outside $Z$) since the event $C_1C_1D^-$
      is covered by a single 2-free pair (either $(\mate{u}{-},\mate{u}{-})$ or
      $(\mate{u}{-},u_{-3})$) and the weight of the event is 9. It remains to
      find a further contribution of $0.125+\eps(u)$ to reach the
      target amount. In particular, we may assume that $u$ is not
      deficient of type II.

      If $\mate{u}{-} \neq v_{+2}$, the event $C_1C_1D^0$ is covered by
      $(v_+,\mate{u}{-})$ and $(\mate{u}{-},u_{-3})$. Using Lemma~\ref{l:dep}, we
      find that $\prob{C_1C_1D^0} \geq 1/2 \cdot 0.5/256$, which is
      sufficient.

      Thus, the present subcase boils down to the situation where
      $\mate{u}{-}$ is adjacent to $v_+$ (i.e., $\mate{u}{-} = v_{+2}$) and $\mate{u}{+} =
      u_{-3}$. Since $u$ is not deficient of type II, it must be that
      $v_-v_{+3}$ is an edge of $M$. In this case, the only events of
      nonzero probability in $\Sigma$ are the events $ABD^-$, $BAD^+$
      and $C_1C_1D^-$ considered above. Fortunately, the condition
      that $v_-v_{+3}\in E(M)$ increases the probability bound
      for $C_1C_1D^+$ from $3/4\cdot 0.5/256$ to $0.5/256$ as
      required. 
    \end{xxcase}

    As all the subcases where $v$ is deficient have been covered in
    Subcase~\ref{subc:ll-notadj}, we may henceforth assume that
    $\eps(u) \leq 0$. In particular, if a further contribution of
    $0.5/256$ can be found (as in the following subcase), then it is
    sufficient.

    \begin{xxcase}{The vertices $\mate{u}{+}$ and $\mate{u}{-}$ are adjacent, $uZv$
        is short and $\mate{u}{+} \neq u_{-3}$.}%
      Suppose first that $\mate{u}{-}$ (and $\mate{u}{+}$) is contained in $Z$. The
      event $C_2C_1D^-$ is then covered by the 1-free pair
      $(\mate{u}{-},u_{-2})$ or $(\mate{u}{+},u_{-2})$. Since its weight is 9, we
      have $\prob{C_2C_1D^-} \geq 1/2 \cdot 0.5/256 = 0.25/256$. Note
      that the event $C_3C_1D^-$ is valid; it is also regular, so
      $\prob{C_3C_1D^-} \geq 0.25/256$. Together, this yields
      $0.5/256$, which is sufficient.

      We may therefore assume that $\mate{u}{-}$ (and $\mate{u}{+}$) are not
      contained in $Z$. The event $C_2C_1D^-$ is covered by
      $(\mate{u}{+},\mate{u}{-})$ or its reverse, each of which is 3-free. By
      Lemma~\ref{l:dep}, $\prob{C_2C_1D^-} \geq 39/40 \cdot 0.5/256 >
      0.48/256$. The event $C_3C_1D^-$, if irregular, has the same
      sensitive pair and it is now 2-free. Since the weight of its
      diagram is $10$, $\prob{C_3C_1D^-} \geq 19/20 \cdot 0.25/256 >
      0.23/256$. The total contribution exceeds the desired $0.5/256$.
    \end{xxcase}

    \begin{figure}
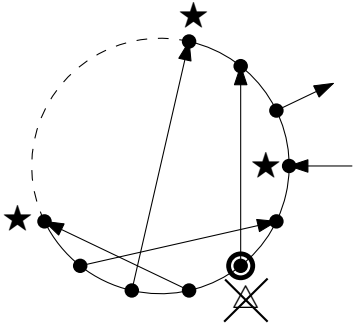

      \centering
      \fig{75}
      \caption{The event $C_1AD^+$ used in Subcase~\ref{subc:ll-last}
        of the proof of Proposition~\ref{p:32}.}
      \label{fig:ll-last}
    \end{figure}

    \begin{xxcase}{The vertices $\mate{u}{+}$ and $\mate{u}{-}$ are adjacent, $uZv$
        is short and $\mate{u}{+} = u_{-3}$.}\label{subc:ll-last}%
      Suppose first that the path $v_+Zu_{-4}$ contains at least two
      vertices distinct from $\mate{v}{-}$. Then the event $C_1AD^+$ (see
      Figure~\ref{fig:ll-last}) is covered by
      $\pair{v_+}{u_{-4}}2$. Since the weight of $C_1AD^+$ is 10, we
      have $\prob{C_1AD^+} \geq 3/4\cdot 0.25/256$. The events
      $C_1C_2D^+$ and $C_1C_3D^+$ have weight 11, but the diagram of
      each of them has a removable symbol at $u_{-3}$, so we get the
      same bound of $3/4\cdot 0.25/256$ for each of $C_1C_2D^+$ and
      $C_1C_3D^+$, since each of the diagrams is covered by one 2-free
      pair. The total contribution is at least $0.56/256$.

      If $\mate{v}{-}$ is the only internal vertex of $v_+Zu_{-4}$, then the
      above events are in fact regular and we obtain an even higher
      contribution. Thus, we may assume that either $v_+$ and $u_{-4}$
      are neighbours on $Z$, or $v_+Zu_{-4}$ contains two internal
      vertices and one of them is $\mate{v}{-}$. 

      The former case is ruled out since we are assuming (from the
      beginning of Case~\ref{case:ll}) that $u$ is not deficient of type
      IIa. It remains to consider the latter possibility. Here, $\mate{v}{-}$
      is either $v_{+2}$ or $v_{+3}$. In fact, it must be $v_{+3}$,
      since otherwise $u$ would be deficient of type I, which has also
      been excluded at the beginning of Case~\ref{case:ll}. But then
      $u$ is deficient of type III, so $\eps(u) = -0.125$. At the same
      time, the unique sensitive pair for each of the events
      $C_1AD^+$, $C_1C_2D^+$ and $C_1C_3D^+$, considered above, is now
      1-free; the probability of the union of these events is thus at
      least $3\cdot 1/2 \cdot 0.25/256 = 0.375/256 = (0.5 +
      \eps(u))/256$ as necessary.
    \end{xxcase}
    
\end{xcase}


  \begin{xcase}{$E(M[U]) = \Setx{u_{-2}v_+,u_{+2}v_-}$.}%
    \label{case:ll-rr}%
    As in Case~\ref{case:ll}, the probability of the event $ABD^-$ is
    at least $1/256$; by symmetry, $\prob{BAD^+} \geq 1/256$. We claim
    that the resulting contribution of $2/256$ is sufficient because
    $\eps(u)\leq 0$. Clearly, $v$ is not of type 0. Applying the
    definitions of the remaining types to $v$, we find that none of
    them is compatible with the presence of the edges $u_{-2}v_+$ and
    $u_{+2}v_-$ in $M$. This shows that $\eps(u)\leq 0$.
  \end{xcase}


  \begin{xcase}{$E(M[U]) = \Setx{u_{-2}u_{+2}}$.}%
    \label{case:low}%
    Recall our assumption that the set $J=\Setx{u_-,u_+,v_-,v_+}$ is
    independent. If we suppose that, moreover, both the paths $uZv$
    and $vZu$ are short, then the mate of each vertex in $J$ must be
    outside $Z$. This means that $\size{\bd Z} = 4$, a contradiction
    with $F$ satisfying the condition in Theorem~\ref{t:ks}. Thus, we
    may assume by symmetry that the path $vZu$ is not short.

    The event $ABD^-$ is regular of weight 9, so $\prob{ABD^-} \geq
    0.5/256$. Similarly, $\prob{BAD^+} \geq 0.5/256$. We need to find
    additional $(1+\eps(u))/256$ to add to the probabilities of
    $ABD^-$ and $BAD^+$ above. Note also that if $v$ is deficient, then
    it must be of type III$^*$ and this only happens in
    Subcase~\ref{subc:down-long}.
      
    \begin{xxcase}{$uZv$ is not short.}\label{subc:down-notshort}%
      Assume that $\mate{u}{+}$ is not contained in $vZu$, and consider the
      events $ABD^+$ and $ABD^0$. If $\mate{v}{-}$ is not contained in $vZu$,
      then $ABD^+$ is covered by the pair $\pair {v_+}{u_{-2}}2$, and
      it follows that $\prob{ABD^+} \geq 3/4 \cdot 0.5/256 =
      0.375/256$. Similarly, $\prob{ABD^0} \geq 0.375/256$. On the
      other hand, if $\mate{v}{-}$ is contained in $vZu$, then the pair
      $(v_+,u_{-2})$ may only be 1-free for $ABD^+$, whence
      $\prob{ABD^+} \geq 1/2 \cdot 0.5/256 = 0.25/256$, but this
      decrease is compensated for by the fact that $\prob{ABD^0} \geq
      0.5/256$ as $ABD^0$ is now regular. Summarizing, if $\mate{u}{+}$ is
      not contained in $vZu$, then the probability of $ABD^+ \cup
      ABD^0$ is at least $0.75/256$.

      The event $BAD^0$ of weight 9 is covered by the pair
      $(u_{+2},v)$, which is 1-free since $uZv$ is not short. Hence,
      $\prob{BAD^0} \geq 1/2 \cdot 0.5/256 = 0.25/256$. Putting this
      together, for $\mate{u}{+}\notin V(vZu)$ we have:
      \begin{multline*}
        \prob{ABD^- \cup BAD^+ \cup ABD^+ \cup ABD^0 \cup BAD^0}\\ 
        \geq
        \frac{0.5 + 0.5 + 0.375 + 0.375 + 0.25}{256} = \frac2{256}.
      \end{multline*}
      Since this is the required amount, we may assume by symmetry
      that $\mate{u}{+}\in V(vZu)$ and $\mate{u}{-}\in V(uZv)$
      (Figure~\ref{fig:down-notshort}).

      If $\mate{v}{+}$ is not contained in $uZv$, then in addition to
      $\prob{BAD^0} \geq 0.25/256$ as noted above, we have
      $\prob{BAD^-} \geq 0.25/256$ for the same reasons. On the other
      hand, $\mate{v}{+}\in V(uZv)$ increases the probability bound for
      $BAD^0$ to $\prob{BAD^0} \geq 0.5/256$ as the event is regular
      in this case. All in all, the contribution of $BAD^- \cup BAD^0$
      is at least $0.5/256$.

      By symmetry, $ABD^+ \cup ABD^0$ also contributes at least
      $0.5/256$. Together with the events $ABD^-$ and $BAD^+$, which
      have each a probability of at least $0.5/256$ as discussed
      above, we have found the required $2/256$.
    \end{xxcase}

    \begin{figure}
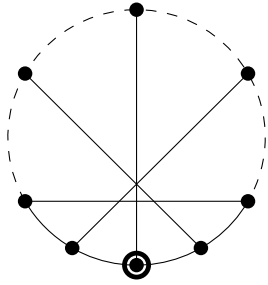

      \centering
      \hf\fig{81}\hf
      \caption{A configuration in Subcase~\ref{subc:down-notshort} of
        the proof of Proposition~\ref{p:32}.}
      \label{fig:down-notshort}
    \end{figure}

    Thus, the path $uZv$ may be assumed to be short. 

    \begin{xxcase}{$\mate{u}{+}\notin V(Z)$.}\label{subc:down-out}%
      As in the previous subcase, $\prob{ABD^+ \cup ABD^0} \geq
      0.75/256$.

      \begin{figure}
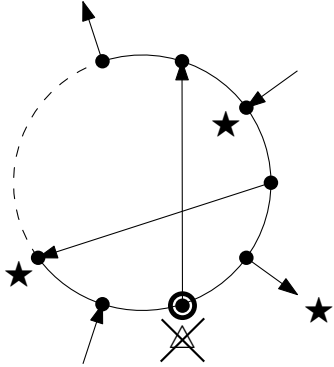

        \centering
        \hspace{4mm}\fig{76}
        \caption{The event $AC_1D^-$ used in
          Subcase~\ref{subc:down-out} of the proof of
          Proposition~\ref{p:32}.}
        \label{fig:down-out}
      \end{figure}

      The event $AC_1D^-$ has weight 10 (see
      Figure~\ref{fig:down-out}). If the cycle of $F$ containing
      $\mate{u}{+}$ is odd, it contains at least 3 vertices different from
      $\mate{u}{+}$ and $\mate{v}{+}$. Thus, $AC_1D^-$ is covered by
      $\pair{\mate{u}{+}}{\mate{u}{+}}3$. By Lemma~\ref{l:dep}, $\prob{AC_1D^-} \geq
      39/40 \cdot 0.25/256 > 0.24/256$.

      Similarly, $AC_1D^0$ has a diagram of weight 10 and is covered
      by $\pair{\mate{u}{+}}{\mate{u}{+}}4$ and $\pair{v_-}{u_{-2}}2$. By
      Lemma~\ref{l:dep}, $\prob{AC_1D^0} \geq 59/80 \cdot 0.25/256 >
      0.18/256$. The probability of $AC_1D^- \cup AC_1D^0$ is thus at
      least $(0.24+0.18)/256 = 0.42/256$, more than the missing
      $0.25/256$.
    \end{xxcase}

    \begin{xxcase}{$\mate{u}{+}\in V(Z)$ and the length of $vZu$ is at least
        7.}\label{subc:down-long}%
      We will show that the assumption about $vZu$ increases the
      contribution of $ABD^+\cup ABD^0$. Suppose that $\mate{v}{-} \in
      V(Z)$. Then $ABD^+$ is covered by $\pair{v_+}{u_{-2}}2$ and
      $ABD^0$ is regular, so $\prob{ABD^+\cup ABD^0} \geq (3/4 + 1)
      \cdot 0.5/256 = 0.875/256$. On the other hand, if $\mate{v}{-}\notin
      V(Z)$, then the pair $(v_+,u_{-2})$ is 3-free for both $ABD^+$
      and $ABD^0$, and we get the same result:
      \begin{equation*}
        \prob{ABD^+\cup ABD^0} \geq 2 \cdot \frac78 \cdot
        \frac{0.5}{256} = 0.875/256.
      \end{equation*}
      We need to find the additional $(0.125 + \eps(u))/256$.

      Suppose first that $v$ is deficient, necessarily of type
      III$^*$, so $\eps(u) = 0.125$. The induced subgraph of $G$ on
      $V(Z)$ is then as shown in Figure~\ref{fig:down-long-deficient};
      in this case, the event $C_1C_1D^-$ is regular and
      $\prob{C_1C_1D^-} \geq 0.5/256$, a sufficient amount.

      \begin{figure}
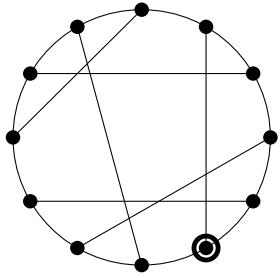

        \centering
        \hf\fig{53}\hf
        \caption{The situation where $v$ is deficient of type III$^*$
          in Subcase~\ref{subc:down-long} of the proof of
          Proposition~\ref{p:32}.}
        \label{fig:down-long-deficient}
      \end{figure}

      We may thus assume that $\eps(u) \leq 0$. Suppose that $\mate{u}{+}$ is
      not adjacent to either $u_{-2}$ or $v_+$. Then the event
      $AC_1D^0$ is covered by $\pair{v_+}{\mate{u}{+}}2$ and
      $\pair{\mate{u}{+}}{u_{-2}}2$. By Lemma~\ref{l:dep}, $\prob{AC_1D^0}
      \geq 1/2 \cdot 0.25/256 = 0.125/256$ as required.

      The vertex $\mate{u}{+}$ can therefore be assumed to be adjacent to
      $u_{-2}$ or $v_+$. The event $C_1C_1D^-$ has only one sensitive
      pair, namely $(\mate{u}{-},\mate{u}{+})$ or its reverse (if $\mate{u}{-} \in V(Z)$)
      or $(\mate{u}{-},\mate{u}{-})$ (if $\mate{u}{-}$ is outside $Z$). If this is a 1-free
      pair, then by Lemma~\ref{l:dep}, $\prob{C_1C_1D^-} \geq 1/2
      \cdot 0.5/256 > 0.125/256$ as required. In the opposite case, it
      must be that $\mate{u}{-}$ is a neighbour of $\mate{u}{+}$. Then, however, we
      observe that $\prob{ABD^+}$ and $\prob{ABD^0}$ are both at least
      $0.5/256$ (as the events are regular), and this increase
      provides the missing $0.125/256$.
    \end{xxcase}

    \begin{figure}
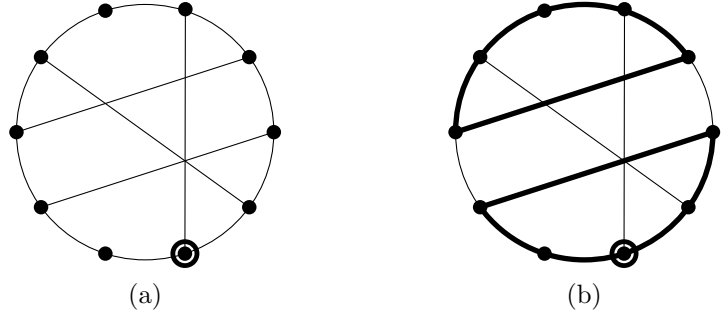

      \centering
      \hf\sfig{101}{}\hf\sfig{102}{}\hf
      \caption{The use of Lemma~\ref{l:split-cycle} in
        Subcase~\ref{subc:down-rest} of the proof of
        Proposition~\ref{p:32}. (a) The cycle $Z$ and its chords
        assuming that $\mate{u}{+} = v_{+2}$. (b) The two 5-cycles (bold)
        contradicting Lemma~\ref{l:split-cycle}(ii). }
      \label{fig:down-rest}
    \end{figure}

    To complete the discussion of Case~\ref{case:low}, it remains to
    consider the following subcase.
    \begin{xxcase}{$uZv$ is short, $\mate{u}{+}\in V(Z)$, and the
        length of $vZu$ is at most 6.}\label{subc:down-rest}%
      Since $vZu$ is not short, its length is $5$ or $6$. Suppose
      first that $\size{vZu} = 5$. By our assumption that
      $\Setx{u_-,u_+,v_-,v_+}$ is independent, $\mate{u}{+} =
      v_{+2}$. Since all the vertices of $Z$ except $u_-$, $v_-$ and
      $v_+$ have their mates in $Z$, we obtain $\size{\bd Z} = 3$,
      contradicting the choice of $F$.

      We may therefore assume that $\size{vZu} = 6$, in which case the
      vertex $\mate{u}{+}$ equals either $v_{+2}$ or
      $v_{+3}$. Consider first the case that $\mate{u}{+} =
      v_{+2}$. Then each edge in $\bd Z$ is incident with a vertex in
      $\Setx{v_{+3},u_-,v_-,v_+}$. By the choice of $F$, $M$ must
      contain an edge with both ends in the latter set. For trivial
      reasons, the only candidate is $v_{+3}v_-$
      (Figure~\ref{fig:down-rest}(a)). However, this is also not an
      edge of $M$ since the 5-cycles $u_{-2}Zu_{+2}$ and $v_-Zv_{+3}$
      would contradict Lemma~\ref{l:split-cycle}(ii). (See
      Figure~\ref{fig:down-rest}(b) for illustration.)

      Thus, $\mate{u}{+} = v_{+3}$. Here, each edge of $\bd Z$ is incident
      with a vertex in $\Setx{u_-,v_-,v_+,v_{+2}}$, and it is easy to
      see that one of these edges must be incident with $v_+$. There are
      two possibilities for an edge with both ends in
      $\Setx{u_-,v_-,v_+,v_{+2}}$, namely $v_-v_{+2}$ or $u_-v_{+2}$. In
      either case, the event $ABD^0$ is easily seen to be regular and
      thus $\prob{ABD^0} \geq 0.5/256$. In fact, this concludes the
      discussion if $v_-v_{+2}\in E(M)$, since then $u$ is deficient of
      type $I$ and $\eps(u) = -0.5$, and the contribution of $0.5/256$
      is sufficient.

      In the remaining case that $u_-v_{+2} \in E(M)$, we need a further
      $0.5/256$, and it is provided by the regular event $ABD^+$.      
    \end{xxcase}
  \end{xcase}


  \begin{xcase}{$E(M[U]) = \Setx{u_{-2}v_-}$.}%
    \label{case:lr}%
    If both the paths $uZv$ and $vZu$ are short, then each edge of
    $\bd Z$ is incident with a vertex in
    $\Setx{u_-,u_+,u_{+2},v_+}$. Our assumptions imply that no edge of
    $M$ joins two of these vertices, so $\size{\bd Z} = 4$ --- a
    contradiction with the choice of $F$. We may therefore assume that
    at least one of $vZu$ and $uZv$ is not short.

    In all the subcases, we can use the regular event $BAD^+$, for
    which we have $\prob{BAD^+} \geq 0.5/256$. Hence, we need to find
    an additional probability of $(1.5 + \eps(u))/256$.

    \begin{xxcase}{$vZu$ is short.}\label{subc:lr-left-short}%
      In this subcase, the path $v_-vv_+$ is contained in a 4-cycle
      and it is not hard to see that $u$ must be deficient of type I
      (neither $uv$ nor $u_-uu_+$ is contained in a 4-cycle, and the
      missing edge $u_{+2}v_+$ rules out cases Ia$^*$ and
      Ib$^*$). Thus, $\eps(u) = -0.5$ and we need to find further
      $1/256$ worth of probability.

      Observe first that by our assumptions, the set
      $\Setx{u_-,u_+,u_{+2},v_+}$ is independent. We will distinguish
      several cases based on whether $\mate{u}{-}$, $\mate{u}{+}$ and $\mate{v}{+}$ are
      contained in $Z$ (and hence in $u_{+3}Zv_{-2}$) or not.

      If $\mate{u}{+} \in V(Z)$, then the events $BAD^0$ and $BAD^-$ are
      regular, and each of them has probability $0.5/256$, which
      provides the necessary $1/256$.

      Suppose thus that $\mate{u}{+}\notin V(Z)$ and consider first the case
      that $\mate{u}{-} \notin V(Z)$. The event $C_1AD^+$ is covered by the
      pair $\pair{\mate{u}{-}}{\mate{u}{-}}4$, so by Lemma~\ref{l:dep} its
      probability is $\prob{C_1AD^+} \geq 79/80 \cdot 0.25/256 >
      0.24/256$. The event $C_1AD^0$ has up to two sensitive pairs: it
      is covered by $\pair{\mate{u}{-}}{\mate{u}{-}}4$ and $\pair{u_{+2}}{v_-}2$,
      where the latter pair is 2-free because $uZv$ is not short. We
      obtain $\prob{C_1AD^0} \geq 59/80 \cdot 0.25/256 > 0.18/256$.

      To find the remaining $0.58/256$ (still for $\mate{u}{-} \notin V(Z)$),
      we use the events $BAD^0$ and $BAD^-$. We claim that their
      probabilities add up to at least $0.75/256$. Indeed, if $\mate{v}{+}
      \notin V(Z)$, then both $BAD^0$ and $BAD^-$ are covered by the
      pair $\pair{u_{+2}}{v_-}2$ (which is 2-free because $uZv$ is not
      short and $\mate{u}{-}\notin V(Z)$). By Lemma~\ref{l:dep}, they have
      probability at least $0.375/256$ each. On the other hand, if
      $\mate{v}{+} \in V(Z)$, then $BAD^0$ is regular and $BAD^-$ is covered
      by $\pair{u_{+2}}{v_-}1$, so $\prob{BAD^0} \geq 0.5/256$ and
      $\prob{BAD^-} \geq 0.25/256$. For both of the possibilities,
      $\prob{BAD^0\cup BAD^-} \geq 0.75/256$ as claimed.

      We can therefore assume that $\mate{u}{-}\in V(Z)$ (and $\mate{u}{+}\notin
      V(Z)$, of course). A large part of the required $1/256$ is
      provided by the event $C_1C_1D^+$, which is covered by the pair
      $\pair{\mate{u}{+}}{\mate{u}{+}}{4}$, so $\prob{C_1C_1D^+} \geq 79/80 \cdot
      0.5/256 > 0.49/256$.

      A final case distinction will be based on the location of
      $\mate{v}{+}$. Suppose first that $\mate{v}{+}\notin V(Z)$. We claim that the
      length of $uZv$ is at least 7. If not, then since $uZv$ is not
      short, the length of $Z$ is 9 or 10. At the same time, $Z$ has
      at least 3 chords (incident with $u$, $u_-$ and $u_{-2}$) and
      therefore $\size{\bd Z} \leq 4$. By the choice of $F$
      and the assumption that the mates of $u_+$ and $v_+$ are outside
      $Z$, $Z$ has length 10 and $\bd Z$ is of size 2. In addition,
      $u_{+2}$ is incident with a chord of $Z$ whose other endvertex
      $w$ is contained in $u_{+3}Zv_{-2}$. However, $\size{uZv} = 6$
      implies that $w\in\Setx{u_{+3},u_{+4}}$, contradicting the
      assumption that $G$ is simple and triangle-free. We conclude
      that $\size{uZv} \geq 7$ as claimed.

      This observation implies that for the event $BAD^0$, the only
      possibly sensitive pair, namely $(u_{+2},v_-)$, is
      2-free. Hence, $\prob{BAD^0} \geq 3/4 \cdot 0.5/256 =
      0.375/256$. Hence, $\prob{BAD^-} \geq 0.375/256$ and this amount
      is sufficient.

      It remains to consider the case that $\mate{v}{+}\in V(Z)$. Being
      regular, the event $BAD^0$ has probability at least
      $0.5/256$. Thus, it is sufficient to find further events forcing
      $u$ of total probability at least $0.01/256$. It is easiest to
      consider the mutual position of $\mate{u}{-}$ and $\mate{v}{+}$ on
      $u_{+3}Zv_{-2}$. If $\mate{u}{-}\in V(\mate{v}{+}Zv_{-2})$, then the event
      $C_1AD^+$ is regular and has probability at least $0.25/256$. In
      the opposite case, $C_1C_1D^0$ is covered by the pair
      $\pair{\mate{u}{+}}{\mate{u}{+}}{4}$, which means that $\prob{C_1C_1D^0} \geq
      79/80 \cdot 0.5/256 > 0.49/256$. In both cases, the probability
      is sufficiently high.
    \end{xxcase}

    Having dealt with Subcase~\ref{subc:lr-left-short}, we can use the
    event $AAD^+$, which is covered by $\pair{v_+}{u_{-2}}2$. By
    Lemma~\ref{l:dep}, $\prob{AAD^+} \geq 3/4 \cdot 0.5/256 =
    0.375/256$ and hence $\prob{BAD^+ \cup AAD^+} \geq
    0.875/256$. Since $v$ is not deficient, we seek a further
    contribution of at least $1.125/256$.

    \begin{figure}
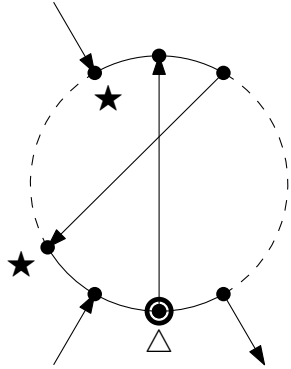

      \centering
      \hf\fig{80}\hf
      \caption{The event $ABD^+$ used in Subcase~\ref{subc:lr-not-short}
        of the proof of Proposition~\ref{p:32}.}
      \label{fig:lr-not-short}
    \end{figure}

    \begin{xxcase}{Neither $vZu$ nor $uZv$ is
        short.}\label{subc:lr-not-short}%
      Consider the event $ABD^+$ of weight 8
      (Figure~\ref{fig:lr-not-short}) which is covered by the pair
      $(v_+,u_{-2})$. Since $vZu$ is not short, the vertices in the
      pair are not neighbours. Furthermore, if the pair is sensitive,
      then the path $v_+Zu_{-2}$ contains at least two internal
      vertices, one of which is different from $\mate{u}{+}$. Thus, the pair
      is 1-free and by Lemma~\ref{l:dep}, $\prob{ABD^+} \geq 1/2 \cdot
      1/256$. If the pair $(v_+,u_{-2})$ is actually 2-free in
      $ABD^+$, then the estimate increases to $3/4 \cdot 1/256$.

      The event $BAD^0$ is covered by the pair $(u_{+2},v_-)$, which
      is 1-free as $uZv$ is not short; moreover, if $\mate{u}{-}\notin
      V(uZv)$, then the pair is 2-free. Thus, $\prob{BAD^0} \geq 1/2
      \cdot 0.5/256 = 0.25/256$ or $3/4 \cdot 0.5/256 = 0.375/256$ in
      the respective cases.

      If the higher estimates hold for both the events $ABD^+$ and
      $BAD^0$ considered above, then the contributions of these events
      total
      \begin{equation*}
        \frac{0.75+0.375}{256} = \frac{1.125}{256},
      \end{equation*}
      which is sufficient.

      Suppose first that we get the higher estimate for
      $\prob{ABD^+}$, that is, that $(v_+,u_{-2})$ is 2-free in
      $ABD^+$. By the above, it may be assumed that $\mate{u}{-}\in V(uZv)$
      and the pair $(u_{+2},v_-)$ is not 2-free in $BAD^0$. We need to
      find an additional $0.125/256$. To this end, we use the event
      $BAD^-$ of weight $9$. The probability of $BAD^-$ is at least
      $1/2 \cdot 0.5/256$ (which is sufficient) if $(u_{+2},v_-)$ is
      1-free in $BAD^-$. This could be false only if $\Setx{\mate{u}{-},\mate{v}{+}}
      = \Setx{u_{+3},v_{-2}}$; for each of the corresponding two
      possibilities, the event $BAD^0$ is a regular one, contradicting
      the assumption that $(u_{+2},v_-)$ is not 2-free in $BAD^0$.

      It remains to discuss the possibility that $(v_+,u_{-2})$ is not
      2-free in $ABD^+$ --- thus, the length of $vZu$ is 6 and
      $\mate{u}{+}\in\Setx{v_{+2},v_{+3}}$. Since the lower bound to
      $\prob{AAD^+}$ increases to $0.5/256$ in this case, the total
      probability of $BAD^+$, $AAD^+$ and $ABD^+$ is at least
      $1.5/256$. In addition, we have a contribution of $1/2 \cdot
      0.5/256$ from $BAD^0$, so we need to add a further $0.25/256$.

      Assume first that $\mate{u}{-}\neq v_{-2}$ and consider the diagram
      $C_1C_1D^-$. We claim that $\prob{C_1C_1D^-} \geq 1/2 \cdot
      0.5/256$. This is certainly true if $\mate{u}{-}\notin V(Z)$ since
      $C_1C_1D^-$ has weight 9 and it is covered by
      $\pair{\mate{u}{-}}{\mate{u}{-}}4$. Suppose thus that $\mate{u}{-}\in V(Z)$. There is
      at most one sensitive pair for $C_1C_1D^-$ ($(\mate{u}{-},\mate{u}{+})$ or
      $(\mate{u}{-},v_-)$ or none). If the event is regular or the sensitive
      pair is 1-free, then $\prob{C_1C_1D^-} \geq 1/2 \cdot 0.5/256$
      as required. Otherwise, since there is only one outgoing arc in
      the diagram for $C_1C_1D^-$, $\mate{u}{-}$ must be adjacent to $v_-$ or
      $\mate{u}{+}$. The former case is ruled out by the assumption $\mate{u}{-}\neq
      v_{-2}$. In the latter case, the 5-cycle $uvZv_{+2}u_-$ and the
      cycle $\mate{u}{+}u_+Zv_-u_{-2}$ provide a contradiction with
      Lemma~\ref{l:split-cycle}(ii).

      We may therefore assume that $\mate{u}{-} = v_{-2}$. Consider the
      cycles $v_-Zu_{-2}$ and $u_-Zv_{-2}$. Since each of the edges
      $v_{-2}v_-$, $u_{-2}u_-$, $uv$ and $u_+\mate{u}{+}$ has one endvertex
      in each of the cycles, Lemma~\ref{l:split-cycle}(i) implies that
      neither $u_{-3}$ nor $v_+$ have their mate in $u_-Zv_{-2}$. We
      claim that $P(BAD^-) \geq 7/8\cdot 0.5/256$. The event is
      covered by the pair $(u_{+2},v_-)$, so by Lemma~\ref{l:dep}, it
      suffices to show that the pair is 3-free. If not, then
      $d_Z(u_{+2},v_-) = 3$ and $u_{+3}$ is the only vertex of
      $u_{+2}Zv_-$ which is not a head of $BAD^-$. In that case,
      however, $\bd Z$ consists of the four edges of $M$ incident with
      a vertex from $\Setx{u_{+2},u_{+3},v_+,v_{+2},v_{+3}} -
      \Setx{\mate{u}{+}}$, contradicting the choice of $F$. We conclude that
      $P(BAD^-) \geq 7/8\cdot 0.5/256$ as claimed. Since this
      contribution exceeds the required $0.25/256$, the discussion of
      Subcase~\ref{subc:lr-not-short} is complete.
    \end{xxcase}

    \begin{xxcase}{$uZv$ is short and either the length of $vZu$ is at
        least 7, or $\mate{u}{+}\notin V(vZu)$.}%
      The event $ABD^+$ is covered by $\pair{v_+}{u_{-2}}2$ by the
      assumption. Thus, $\prob{ABD^+} \geq 3/4 \cdot 1/256$. In view
      of the events $BAD^+$ (probability at least $0.5/256$) and
      $AAD^+$ (probability at least $3/4 \cdot 0.5/256$), we need to
      collect further $0.375/256$.

      Suppose first that $\mate{u}{+}\notin V(vZu)$. The event $AC_1D^+$ of
      weight 9 is covered by $\pair{\mate{u}{+}}{\mate{u}{+}}4$ and
      $\pair{v_+}{u_{-2}}2$. By~Lemma~\ref{l:dep}, $\prob{AC_1D^+}
      \geq 59/80 \cdot 0.5/256 > 0.36/256$. The event $AC_2D^+$ of
      weight 11 is covered by $\pair{v_+}{u_{-2}}2$; thus,
      $\prob{AC_2D^+} \geq 3/4 \cdot 0.125/256$, which together with
      $\prob{AC_1D^+}$ yields more than the required $0.375/256$.

      We may therefore assume that $\mate{u}{+}\in V(vZu)$, which increases
      $\prob{AAD^+}$ to at least $0.5/256$ (so the missing probability
      is now $0.25/256$).

      Suppose that $\mate{u}{-}$ and $\mate{u}{+}$ are non-adjacent. If $\mate{u}{-}\notin
      V(Z)$, then $C_1C_1D^-$ is covered by
      $\pair{\mate{u}{-}}{\mate{u}{-}}3$. Otherwise, it is covered by
      $\pair{\mate{u}{-}}{\mate{u}{+}}1$ (we have to consider $\mate{v}{+}$ here). In
      either case, $\prob{C_1C_1D^-} \geq 1/2 \cdot 0.5/256$ as
      required.

      We may thus assume that $\mate{u}{-}$ and $\mate{u}{+}$ are adjacent. The
      event $AC_1D^+$ has weight 9 and at most one possibly sensitive
      pair; this pair is $(\mate{u}{+},u_{-2})$ if $(\mate{u}{-})_+ = \mate{u}{+}$, or
      $(\mate{u}{+},v_+)$ otherwise. If the sensitive pair is 2-free, we are
      done since $\prob{AC_1D^+} \geq 3/4 \cdot 0.5/256$. In the
      opposite case, we get two possibilities.

      The first possibility is that $\mate{u}{+}$ is adjacent to $u_{-2}$, so
      $\mate{u}{+} = u_{-3}$. In this case, the 5-cycle $u_{-3}u_+Zv_-u_{-2}$
      and the cycle $uvZu_{-4}u_-$ provide a contradiction with
      Lemma~\ref{l:split-cycle}(ii).

      The second possibility is that $\mate{u}{+}$ is adjacent to $v_+$,
      i.e., $\mate{u}{+} = v_{+2}$. Here, the event $AC_2D^+$ is regular, and
      $\prob{AC_2D^+} \geq 0.25/256$ as desired.
    \end{xxcase}

    \begin{xxcase}{$uZv$ is short, the length of $vZu$ is 6, and
        $\mate{u}{+}\in V(vZu)$.}%
      The vertex $\mate{u}{+}$ equals either $v_{+2}$ or $v_{+3}$. Each of
      the events $ABD^+$, $BAD^+$, $AAD^+$ (considered earlier) now
      have probability at least $0.5/256$. We need to find an
      additional $0.5/256$.

      If $\mate{u}{+} = v_{+2}$, then each edge of $\bd Z$ is incident with a
      vertex from the set $\Setx{u_-,u_{+2},v_+,v_{+3}}$. By the
      choice of $F$, some edge of $M$ must join two of these vertices;
      our assumptions imply that the only candidate is the edge
      $u_{-3}u_{+2}$. The events $AC_2D^+$, $C_2AD^+$ and $C_2C_2D^+$
      are regular, with $AC_2D^+$ having a removable symbol, and their
      probabilities are easily computed to be at least $0.25/256$,
      $0.125/256$ and $0.125/256$, respectively. This adds up to the
      required $0.5/256$.

      On the other hand, if $\mate{u}{+} = v_{+3}$, then each edge of $\bd Z$
      is incident with $\Setx{u_-,u_{+2},v_+,v_{+2}}$. In two of the
      cases, there is a pair of 5-cycles which yields a contradiction
      with Lemma~\ref{l:split-cycle}(ii): if $u_{+2}v_{+2} \in E(M)$,
      then the cycles are $u_{-3}Zu_+$ and $u_{+2}Zv_{+2}$, while if
      $u_-v_{+2} \in E(M)$, then the cycles are $u_-uvZv_{+2}$ and
      $u_+2Zv_-u_{-2}u_{-3}$. All the other cases are ruled out by the
      assumptions (notably, the assumption that $u_{+2}v_+ \notin
      E(M)$).
    \end{xxcase}

    The only possibility in Case~\ref{case:lr} not covered by the
    above subcases is that $uZv$ is short, $vZu$ has length 5 and
    $\mate{u}{+}\in V(vZu)$. This is, however, excluded by our choice of $Z$:
    the cycle $Z$ of length 9 would have at least three chords,
    implying $\size{\bd Z}\in\Setx{1,3}$, which is impossible.
  \end{xcase}


  \begin{xcase}{$E(M[U]) =
      \Setx{u_{-2}v_-,u_{+2}v_+}$.}\label{case:cross}%
    We will call a chord $f$ of $Z$ \emph{bad} if
    $f\in\Setx{u_-u_{+3},u_+u_{-3},u_-v_{-2},u_+v_{+2}}$.

    \begin{xxcase}{Neither $uZv$ nor $vZu$ is short and $Z$ has no bad
      chord.}\label{subc:8-noshort-nobad}%
    The event $ABD^+$ has one sensitive pair, namely $(v_+,u_{-2})$
    (see Figure~\ref{fig:8-noshort-nobad}). We claim that this pair is
    2-free. Suppose not; then it must be that $\mate{u}{+}$ is an internal
    vertex of $v_+Zu_{-2}$ and there is exactly one other internal
    vertex in the path. This would mean that the edge of $M$ incident
    with $u_+$ is a bad chord, contrary to the assumption. Hence
    $(v_+,u_{-2})$ is 2-free in $ABD^+$ and $\prob{ABD^+}\geq 3/4
    \cdot 1/256$ as the weight of $ABD^+$ is 8.

    \begin{figure}
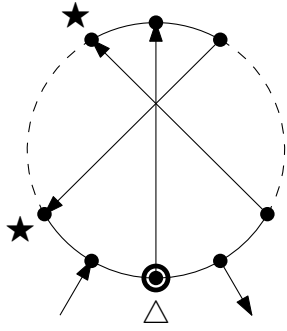

      \centering
      \hf\fig{98}\hf
      \caption{The event $ABD^+$ used in
        Subcase~\ref{subc:8-noshort-nobad} of the proof of
        Proposition~\ref{p:32}.}
      \label{fig:8-noshort-nobad}
    \end{figure}

    For a similar reason (using the symmetry in the definition of a
    bad chord), $\prob{BAD^-} \geq 3/4 \cdot 1/256$. Since $v$ is not
    deficient in this subcase, it suffices to find a further $0.5/256$
    to reach the desired bound.
    
    Suppose first that $\mate{u}{-}$ and $\mate{u}{+}$ are not neighbours. 
    
    If $\mate{u}{-}$ and $\mate{u}{+}$ are contained in two distinct cycles of $F$,
    both different from $Z$, then by Lemma~\ref{l:dep}, we have
    $\prob{C_1C_1D^+} \geq 39/40 \cdot 0.5/256$ and the same estimate
    holds for $C_1C_1D^0$ and $C_1C_1D^-$. Thus
    \begin{equation*}
      \prob{C_1C_1D^+ \cup C_1C_1D^0 \cup C_1C_1D^-} \geq \frac{1.46}{256},
    \end{equation*}
    much more than the required amount.

    If $\mate{u}{+}$ and $\mate{u}{-}$ are contained in the same cycle $Z'\neq Z$ of
    $F$, then the event $C_1C_1D^+$ is covered by $\pair{\mate{u}{+}}{\mate{u}{-}}2$
    and $\pair{\mate{u}{-}}{\mate{u}{+}}2$. By Lemma~\ref{l:dep}, $\prob{C_1C_1D^+}
    \geq 1/2 \cdot 0.5/256$. Since the same holds for $C_1C_1D^0$ and
    $C_1C_1D^-$, we find a sufficient contribution of $0.75/256$.

    If, say, $\mate{u}{+}$ is contained in $Z$ and $\mate{u}{-}$ is not, then
    $C_1C_1D^+$ is covered by the pairs $\pair{v_+}{\mate{u}{+}}{2}$ and
    $\pair{\mate{u}{-}}{\mate{u}{-}}{4}$ (note that the first pair is 2-free since
    $\mate{u}{+}\neq v_{+2}$ by the absence of bad chords). Using
    Lemma~\ref{l:dep}, we find that $\prob{C_1C_1D^+} \geq 59/80 \cdot
    0.5/256 > 0.36/256$. Similarly, $C_1AD^-$ is covered by
    $\pair{u_{+2}}{v_-}{2}$ and $\pair{\mate{u}{-}}{\mate{u}{-}}{4}$, so by
    Lemma~\ref{l:dep}, $\prob{C_1AD^-} \geq 59/80\cdot 0.5/256 >
    0.36/256$. Thus,
    \begin{equation*}
      \prob{C_1C_1D^+ \cup C_1AD^-} \geq \frac{0.36 + 0.36}{256} =
      \frac{0.72}{256}
    \end{equation*}
    and we are done.

    Thus, still in the case that $\mate{u}{+}$ and $\mate{u}{-}$ are not neighbours,
    we may assume that they are both contained in $Z$. Consider the
    event $AC_1D^+$. If $\mate{u}{-} \in V(vZu)$, then the event is covered
    by a single 2-free pair, namely $(v_+,\mate{u}{+})$ or $(\mate{u}{+},u_{-2})$,
    so $\prob{AC_1D^+} \geq 3/4 \cdot 0.5/256$. On the other hand, if
    $\mate{u}{-}\in V(uZv)$, then $AC_1D^+$ is regular if $\mate{u}{+}\in V(uZv)$
    or covered by $\pair{v_+}{\mate{u}{+}} 2$ and $\pair{\mate{u}{+}}{u_{-2}}2$
    otherwise. Summing up, $\prob{AC_1D^+} \geq 1/2\cdot
    0.5/256$. Symmetrically, $\prob{C_1AD^-} \geq 1/2\cdot 0.5/256$
    and we have found the necessary $0.5/256$.

    We may thus assume that $\mate{u}{-}$ and $\mate{u}{+}$ are neighbours.

    If they are in contained in a cycle of $F$ different from $Z$,
    then the event $C_1AD^-$ is covered by the 2-free pair
    $(u_{+2},v_-)$, so $\prob{C_1AD^-}\geq 3/4 \cdot 0.5/256$. By
    symmetry, $\prob{AC_1D^+} \geq 3/4\cdot 0.5/256$, making for a
    sufficient contribution of $1.5/256$.

    We may thus suppose that $\mate{u}{-}$ and $\mate{u}{+}$ are both contained in
    $vZu$. By the absence of bad chords, $\mate{u}{+}$ is not a neighbour of
    $v_+$ nor $u_{-2}$. Thus, the event $AC_1D^+$ is covered by a
    single 2-free pair, namely $(v_+,\mate{u}{+})$ or $(\mate{u}{+},u_{-2})$, and
    $\prob{AC_1D^+} \geq 3/4 \cdot 0.5/256$. Moreover, $\prob{C_1AD^-}
    \geq 3/4\cdot 0.5/256$ since the event is covered by $\pair
    {u_{+2}} {v_-} 1$, so
    \begin{equation*}
      \prob{AC_1D^+ \cup C_1AD^-} \geq \frac{0.375 + 0.375}{256} =
      \frac{0.75}{256}
    \end{equation*}
    as required. This finishes Subcase~\ref{subc:8-noshort-nobad}.
    \end{xxcase}

    \begin{figure}
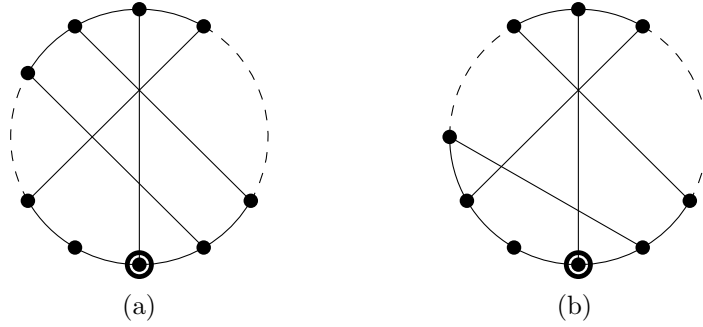

      \centering
      \hf\sfig{99}{}\hf\sfig{100}{}\hf
      \caption{The possibilities in Subcase~\ref{subc:8-noshort-bad}
        of the proof of Proposition~\ref{p:32}.}
      \label{fig:8-noshort-bad}
    \end{figure}

    \begin{xxcase}{Neither $uZv$ nor $vZu$ is short, but $Z$ has a bad
        chord.}\label{subc:8-noshort-bad}%
      By symmetry, we may assume that $u_+v_{+2}$ or $u_+u_{-3}$ is a
      bad chord of $Z$ (see Figure~\ref{fig:8-noshort-bad}).

      Consider first the possibility that $u_+v_{+2} \in E(M)$. By
      Lemma~\ref{l:split-cycle}(i), $uv$ and $u_{-2}v_+$ are the only
      two chords of $Z$ with one endvertex in $v_{+3}Zu$ and the other
      one in $u_{+3}Zv$. In particular, $\mate{u}{-}\notin V(uZv)$.

      We will use the events $BAD^-$, $ABD^+$, $C_1C_1D^-$ and
      $C_1AD^-$. Let us estimate their probabilities. The event
      $BAD^-$ of weight 8 is covered by the pair $(u_{+2},v_-)$, which
      is 2-free as $uZv$ is not short. Thus, $\prob{BAD^-} \geq 3/4
      \cdot 1/256$ by Lemma~\ref{l:dep}. Similarly, $ABD^+$ is covered
      by the pair $\pair{v^+}{u_{-2}}{1}$ and therefore $\prob{ABD^+}
      \geq 1/2\cdot 1/256$. The event $C_1C_1D^-$ of weight 9 is
      covered by $\pair{\mate{u}{-}}{\mate{u}{-}}{4}$, implying $\prob{C_1C_1D^-}
      \geq 79/80 \cdot 0.5/256$. Finally, the event $C_1AD^-$ of
      weight 9 is covered by the pairs $\pair{\mate{u}{-}}{\mate{u}{-}}{4}$ and
      $\pair{u_{+2}}{v_-}{2}$ (the latter of which is, again, 2-free
      since $uZv$ is not short). By Lemma~\ref{l:dep}, $\prob{C_1AD^-}
      \geq 59/80 \cdot 0.5/256$. Summarizing,
      \begin{equation*}
        \prob{BAD^-\cup ABD^+\cup C_1C_1D^-\cup C_1AD^-} >
        \frac{0.75 + 0.5 + 0.49 + 0.36}{256} = \frac{2.1}{256},
      \end{equation*}
      which is sufficient.

      We may therefore assume that $u_+u_{-3}$ is a bad chord
      (Figure~\ref{fig:8-noshort-bad}(b)). The length of $vZu$ is at
      least 6, as can be seen by considering the cycles $u_{-3}Zu_+$
      and $u_{+2}Zv_+$ and using
      Lemma~\ref{l:split-cycle}(ii). Furthermore,
      Lemma~\ref{l:split-cycle}(i) implies that $\mate{u}{-}\neq u_{-4}$,
      since otherwise the cycles $u_+Zv_-u_{-2}u_{-3}$ and
      $u_-uvZu_{-4}$ would provide a contradiction.

      We distinguish three cases based on the position of
      $\mate{u}{-}$. Assume first that $\mate{u}{-}$ is contained in $vZu$. The
      regular event $ABD^+$ has probability at least $1/256$. The event
      $BAD^-$ is covered by the pair $(u_{+2},v_-)$ which is 2-free
      since $uZv$ is not short. Thus $\prob{BAD^-}\geq 3/4\cdot
      1/256$. Finally, the event $C_1C_1D^-$ is covered by the pair
      $(\mate{u}{-},u_{-3})$ which is 2-free since $\mate{u}{-}\neq u_{-4}$ as noted
      above. Consequently,
      \begin{equation*}
        \prob{ABD^+ \cup BAD^- \cup C_1C_1D^-} \geq \frac1{256} +
        \frac{0.75}{256} + \frac{0.75\cdot 0.5}{256} =\frac{2.125}{256},
      \end{equation*}
      more than the required $2/256$.

      Suppose next that $\mate{u}{-}$ is contained in $uZv$. Note that $\mate{u}{-}
      \neq v_{-2}$ by Lemma~\ref{l:split-cycle}(i). Since the event
      $ABD^+$ is covered by $\pair{v_+}{u_{-2}}{1}$, $\prob{ABD^+}
      \geq 1/2 \cdot 1/256$. Similarly, $BAD^-$ is covered by
      $\pair{u_{+2}}{v_-}{1}$ and so $\prob{BAD^-} \geq 1/2 \cdot
      1/256$. The event $C_1C_1D^-$ is covered by the 2-free pair
      $(v_+,u_{-3})$ and thus $\prob{C_1C_1D^-}\geq 3/4 \cdot 0.5/256
      = 0.375/256$. The same bound is valid for $C_1C_1D^+$. Finally,
      $\prob{C_1C_1D^0} \geq 1/2 \cdot 0.5/256$ as the event is
      covered by $\pair{\mate{u}{-}}{v_-}{2}$ and
      $\pair{v_+}{u_{-3}}{2}$. Altogether, we have
      \begin{multline*}
        \prob{ABD^+ \cup BAD^- \cup C_1C_1D^- \cup C_1C_1D^+ \cup
          C_1C_1D^0} \geq\\\frac{0.5+0.5+0.375+0.375+0.25}{256} = \frac2{256}.
      \end{multline*}

      The last remaining possibility is that $\mate{u}{-}$ is not contained
      in $Z$. We have $\prob{ABD^+} \geq 0.5/256$ and $\prob{BAD^-}
      \geq 0.75/256$ by standard arguments. The event $C_1C_1D^-$ is
      covered by the pair $\pair{\mate{u}{-}}{\mate{u}{-}}{4}$, so
      $\prob{C_1C_1D^-}\geq 79/80 \cdot 0.5/256 > 0.49/256$ by
      Lemma~\ref{l:dep}. Similarly, $\prob{C_1C_1D^+} \geq 59/80 \cdot
      0.5/256 > 0.36/256$ since the event is covered by
      $\pair{\mate{u}{-}}{\mate{u}{-}}{4}$ and $\pair{v_+}{u_{-3}}{2}$. The total
      contribution is at least $2.1/256$. This concludes
      Subcase~\ref{subc:8-noshort-bad}.
    \end{xxcase}

    We may now assume that the path $uZv$ is short; note that this
    means that $u$ is deficient of type Ia or Ib. In the former case,
    there is nothing to prove as $2 + \eps(u) = 0$. Therefore, suppose
    that $u$ is of type Ib (i.e., $\mate{u}{+} = v_{+2}$). Since $\eps(u) =
    -1.5$, it remains to find events forcing $u$ with total
    probability at least $0.5/256$. It is sufficient to consider the
    event $ABD^+$ of weight 8, which is covered by the 1-free pair
    $(v_+,u_{-2})$, and therefore $\prob{ABD^+} \geq 0.5/256$ by
    Lemma~\ref{l:dep}. This finishes the proof of
    Case~\ref{case:cross} and the whole proposition.
  \end{xcase}
\end{proof}


\section{Augmentation}
\label{sec:phase5}

In this section, we show that it is possible to apply the augmentation
step mentioned in the preceding sections. 

Suppose that $u$ is a deficient vertex of $G$ and $v=u'$. Let us
continue to use $Z$ to denote the cycle of the 2-factor $F$ containing
$u$. The \emph{sponsor} $s(u)$ of $u$ is one of its neighbours,
defined as follows:
\begin{itemize}
\item if $u$ is deficient of type 0 (recall that this type was defined
  at the beginning of Section~\ref{sec:nochord}), then $s(u)$ is the
  $F$-neighbour $u$ with $\eps(s(u))=1$; if there are two such
  $F$-neighbours, we choose $s(u)=u_-$,
\item if $u$ is deficient of any other type (in particular, $v\in
  V(Z)$), then $s(u) = v$.
\end{itemize}

\begin{observation}
  Every vertex is the sponsor of at most one other vertex.
\end{observation}
\begin{proof}
  Clearly, a given vertex can only sponsor its own neighbours, that
  is, its mate and $F$-neighbours. Suppose that $u$ is the sponsor of
  its mate $v$; thus, $u\in C_v$. Suppose also that $u$ is the sponsor
  of one of its $F$-neighbours, say $u_+$. Then $uv$ belongs to a
  4-cycle intersecting $C_{\mate{u}{+}}$, but this is not possible since
  $C_{\mate{u}{+}} \neq C_v$.
 
  The only remaining possibility is that $u$ is the sponsor of both of
  its $F$-neighbours. In that case, both $u_+$ and $u_-$ are deficient
  of type 0 and $\eps(u) = 1$. Thus, $uv$ is contained in a 4-cycle,
  but neither $u_+$ or $u_-$ is, giving rise to a contradiction.
\end{proof}

Recall that $N[u]$ denotes the closed neighbourhood of $u$, i.e.,
$N[u] = N(u)\cup \Setx u$. An independent set $J$ in $G$ is
\emph{favourable} for $u$ if $N[u]\cap J = \Setx{s(u)}$. The
\emph{receptivity} of $u$, denoted $\rho(u)$, is the probability that
a random independent set (with respect to the distribution given by
Algorithm 1) is favourable for $u$. We say that $u$ is
\emph{$k$-receptive} ($k\geq 0$) if the receptivity of $u$ is at least
$k/256$.

For an independent set $J$, we let $p(J)$ denote the probability that
the random independent set produced by Algorithm 1 is equal to $J$.
We fix an ordering $J_1,\dots,J_s$ of all independent sets $J$ in $G$
such that $p(J) > 0$. Furthermore, an ordering $u_1,\dots,u_r$ of all
deficient vertices is chosen in such a way that $\size{\eps(u_i)} \leq
\size{\eps(u_j)}$ if $1\leq i < j \leq s$ (to which we refer as the
\emph{monotonicity} of the ordering).

Let $u_i$ be a deficient vertex. We let $\nbrx{u_i}$ be the set of all
deficient neighbours $u_j$ of $u_i$ such that $j < i$; furthermore, we
put $\nbrxc{u_i} = \nbrx{u_i} \cup \Setx{u_i}$. We define $\eta(u_i)$
as
\begin{equation*}
  \eta(u_i) = \sum_{u_j\in\nbrxc{u_i}} \size{\eps(u_j)}.
\end{equation*}

We aim to replace $s(u_i)$ with $u_i$ in some of the independent sets
that are favourable for $u_i$, thereby boosting the probability of the
inclusion of $u_i$ in the random independent set $I$. Clearly, this
requires that the receptivity of $u_i$ is at least $\size{\eps(u_i)}/256$,
for otherwise the probability of $u_i\in I$ cannot be increased to the
required $88/256$ in this way. We also need to take into account the
fact that an independent set may be favourable for $u_i$ and its
neighbour at the same time, but the replacement can only take place
once. To dispatch the replacements in a consistent way, the following
lemma will be useful. We remark that the number $p(u_i,J_j)$ which
appears in the statement will turn out to be the probability that
$u_i$ is added to the random independent set during Phase 5 of the
execution of the algorithm.

\begin{lemma}
  \label{l:events}
  If the receptivity of each deficient vertex $u_i$ is at least
  $\eta(u_i)$, then we can choose a nonnegative real number
  $p(u_i,J_j)$ for each deficient vertex $u_i$ and each independent
  set $J_j$ in such a way that the following holds:
  \begin{enumerate}[\quad(i)]
  \item $p(u_i,J_j) = 0$ whenever $J_j$ is not favourable for $u_i$,
  \item for each deficient vertex $u_i$, $\sum_j p(u_i,J_j)\cdot
    p(J_j) = \size{\eps(u_i)}/256$,
  \item for each independent set $J_j$ and deficient vertex $u_i$,
    $\sum_{u_t\in \nbrxc{u_i}} p(u_t,J_j) \leq 1$.
  \end{enumerate}
\end{lemma}
\begin{proof}
  We may view the numbers $p(u_i,J_j)$ as arranged in a matrix (with
  rows corresponding to vertices) and choose them in a simple greedy
  manner as follows. For each $i=1,\dots,r$ in this order, we
  determine $p(u_i,J_1)$, $p(u_i,J_2)$ and so on. Let $\vec{r_i}$ be
  the $i$-th row of the matrix, with zeros for the entries that are
  yet to be determined. Furthermore, let $\vec p =
  (p(J_1),\dots,p(J_s))$.

  For each $i,j$ such that $J_j$ is favourable for $u_i$, $p(u_i,J_j)$
  is chosen as the maximal number such that $\vec{r_i}\cdot\vectrans p
  \leq \size{\eps(u_i)}/256$, and its sum with any number in the
  $j$-th column corresponding to a vertex in $\nbrx{u_i}$ is at most
  one. In other words, we set
  \begin{equation}\label{eq:prob}
    p(u_i,J_j) =
    \min\Bigl(\frac{\size{\eps(u_i)}/256 - \sum_{\ell=1}^{j-1}
      p(u_i,J_\ell)\cdot p(J_\ell)}{p(J_j)},
    1 - \sum_{u_\ell\in\nbrx{u_i}} p(u_\ell,J_j)\Bigr)
  \end{equation}
  if $J_j$ is favourable for $u_i$, and $p(u_i,J_j) = 0$
  otherwise. Note that the denominator in the fraction is nonzero
  since every independent set $J_j$ with $1\leq j\leq s$ has $p(J_j) >
  0$. By the construction, properties (i) and (iii) in the lemma are
  satisfied, and so is the inequality $\vec{r_i} \cdot \vectrans p
  \leq \size{\eps(u_i)}/256$ in property (ii). We need to prove the
  converse inequality.

  Suppose that for some $i$, $\vec{r_i} \cdot \vectrans p$ is strictly
  smaller than $\size{\eps(u_i)}/256$. This means that
  in~\eqref{eq:prob}, for each $j$ such that $J_j$ is favourable for
  $u_i$, $p(u_i,J_j)$ equals the second term in the outermost pair of
  brackets. In other words, for each such $j$, we have
  \begin{equation*}
    \sum_{u_\ell\in\nbrxc{u_i}} p(u_\ell,J_j) = 1.
  \end{equation*}
  Thus, we can write
  \begin{align}\label{eq:rowsums}
    \sum_{\stackrel{j}{J_j\text{ favourable for }u_i}} \Bigl(
    \sum_{u_\ell\in\nbrxc{u_i}} p(u_\ell,J_j) \Bigr) \cdot p(J_j) &=
    \sum_{\stackrel{j}{J_j\text{ favourable for }u_i}} p(J_j) \\
    &=\rho(u_i) \geq \eta(u_i) = \sum_{u_\ell\in
      \nbrxc{u_i}}\frac{\size{\eps(u_\ell)}}{256},\notag
  \end{align}
  where the inequality on the second line follows from our assumption
  on the receptivity of $u_i$.

  On the other hand, the expression on the first line of
  \eqref{eq:rowsums} is dominated by the sum of the scalar products of
  $\vec p$ with the rows corresponding to vertices in
  $\nbrxc{u_i}$. For each such vertex $u_\ell$, we know from the first
  part of the proof that $\vec{r_\ell}\cdot \vectrans p \leq
  \size{\eps(u_\ell)}/256$. Comparing with \eqref{eq:rowsums}, we find
  that we must actually have equality both here and in
  \eqref{eq:rowsums}; in particular,
  \begin{equation*}
    \vec{r_i}\cdot\vectrans p = \frac{\size{\eps(u_i)}}{256},
  \end{equation*}
  a contradiction.
\end{proof}

For brevity, we will say that an event $X\subseteq\Omega$ is
\emph{favourable for $u$} if the independent set $I(\sigma)$ is
favourable for $u$ for every situation $\sigma\in X$. We lower-bound
the receptivity of deficient vertices as follows:

\begin{proposition}
  \label{p:receptive}
  Let $u$ be a deficient vertex. The following holds:
  \begin{enumerate}[\quad(i)]
  \item $u$ is $1.9$-receptive,
  \item if $u$ is of type 0, then it is $3$-receptive,
  \item if $u$ is of type Ia or Ib (or their mirror types), then it
    is $8$-receptive.
  \end{enumerate}
\end{proposition}
\begin{proof}
  All the event(s) discussed in this proof will be favourable for $u$,
  as it is easy to check. To avoid repetition, we shall not state this
  property in each of the cases.

  (i) First, let $u$ be a deficient vertex of type I. We distinguish
  three cases, in each case presenting an event which is favourable
  for $u$ and has sufficient probability. If $u_{-2}u_{+2}$ is not an
  edge of $M$, then the event $Q_1$ given by the diagram in
  Figure~\ref{fig:receptive-i}(a) is valid. Since it is a regular
  diagram of weight 7, $\prob{Q_1} \geq 2/256$ by
  Lemma~\ref{l:dep}. Thus, $\rho(u)\geq 2/256$ as $Q_1$ is favourable
  for $u$.

  We may thus assume that $u_{-2}u_{+2}\in E(M)$. Suppose that neither
  $\mate{u}{-}$ nor $\mate{u}{+}$ is contained in $vZu$. Consider the event $Q_2$,
  given by the diagram in Figure~\ref{fig:receptive-i}(b). Since the
  edge $uv$ is not contained in a 4-cycle ($u$ being deficient),
  neither $v_-$ nor $v_+$ is the mate of $u_+$, so the diagram is
  valid. The event is covered by the pair $(\mate{u}{+},\mate{u}{+})$. If the pair
  is sensitive, then the cycle of $F$ containing $\mate{u}{+}$ has length at
  least 5, and hence it contains at least two vertices different from
  $\mate{u}{+}$, $\mate{v}{-}$ and $\mate{v}{+}$. Thus, the pair is 2-free, and we have
  $\prob{Q_2} \geq 19/20 \cdot 2/256 = 1.9/256$ by Lemma~\ref{l:dep}.

  By symmetry, we may assume that each of $uZv$ and $vZu$ contain one
  of $\mate{u}{-}$ and $\mate{u}{+}$. Hence, the event $Q_3$, defined by
  Figure~\ref{fig:receptive-i}(c), is regular and $\prob{Q_3} \geq
  2/256$. (The event is valid for the same reason as $Q_2$.)

  To finish part (i), it remains to discuss deficient vertices of
  types other than I. In view of parts (ii) and (iii), it suffices to
  look at types II, IIa, III and their mirror variants. Each of these
  types is consistent with the diagram in
  Figure~\ref{fig:receptive-i}(d) or its symmetric version. The diagram
  of weight 6 defines a regular event $Q_4$, whose probability is at
  least $4/256$ by Lemma~\ref{l:dep}. This proves part (i).

  \begin{figure}
    \centering%
    \sfigdef{84}%
    \hf\sfigtop{82}{$Q_1$.}\hf\sfigtop{84}{$Q_2$.}\hf\\
    \sfigdef{83}
    \hspace{4mm}\hf\sfigtop{83}{$Q_3$ (note that each of $uZv$, $vZu$ contains one of
      $\mate{u}{+}$, $\mate{u}{-}$).}\hf\hspace{4mm}
    \subfloat[$Q_4.$]{\raisebox{2em}{\fig{56}}}\hspace{-10mm}\hf
    \caption{Events used in the proof of
      Proposition~\ref{p:receptive}(i).}
    \label{fig:receptive-i}
  \end{figure}

  We prove (ii). Let $u$ be deficient of type 0. We may assume that
  $u_-$ is contained in a 4-cycle intersecting the cycle $C_v$; in
  particular, the mates of $u_-$ and $u_{-2}$ are contained in
  $C_v$. By the definition of type 0, we also know that neither
  $u_{-2}$ nor $u_{+2}$ has a neighbour in $\Setx{v_-,v_+}$.
  
  Suppose that the set $\Setx{u_{-2},u_{+2},v_-,v_+}$ is
  independent. Since $\mate{u}{-}\in V(C_v)$, the event $R$ defined by the
  diagram in Figure~\ref{fig:receptive-ii}(a) is regular and it is easy
  to see that it is favourable for $u$ and its probability is at least
  $1/256$. Since the same holds for the events $R^+$ and $R^-$
  obtained by reversing the arrow at $v_-$ or $v_+$, respectively, we
  have shown that $u$ is 3-receptive in this case.
  
  If $M$ includes the edge $u_{-2}v_+$, then both $R$ and $R^+$ remain
  valid events, and the probability of each of them increases to at
  least $2/256$, showing that $u$ is 4-receptive. An analogous
  argument applies if $M$ includes $u_{-2}v_-$.

  It remains to consider the possibility that $u_{+2}v_-$ or
  $u_{+2}v_+$ is in $M$. Suppose that $u_{+2}v_- \in E(M)$. The event
  $R^-$ remains valid and regular; its probability increases to at
  least $2/256$. Let $S^+$ and $T^+$ be the events given by diagrams
  in Figure~\ref{fig:receptive-ii}(b) and (c), respectively. It is easy
  to check that $R^+$, $S^+$ and $T^+$ are pairwise disjoint and
  favourable for $u$. The event $S^+$ is covered by the pair
  $\pair{u_{+2}}{u_-}1$ and Lemma~\ref{l:dep} implies that $\prob{S^+}
  \geq 0.5/256$. The event $T^+$ is regular and $\prob{T^+} \geq
  0.5/256$. Since $\prob{R^+\cup S^+\cup T^+} \geq 3/256$, $u$ is
  3-receptive.

  In the last remaining case, namely $u_{+2}v_+ \in E(M)$, we argue
  similarly. Let $S^-$ and $T^-$ be the events obtained by reversing
  both arcs incident with $v_+$ and $v_-$ in the diagram for $S^+$ or
  $T^+$, respectively. It is routine to check that $\prob{R^-\cup
    S^-\cup T^-} \geq 3/256$ and the events are favourable for
  $u$. Hence, $u$ is 3-receptive. The proof is finished.
  
  \begin{figure}
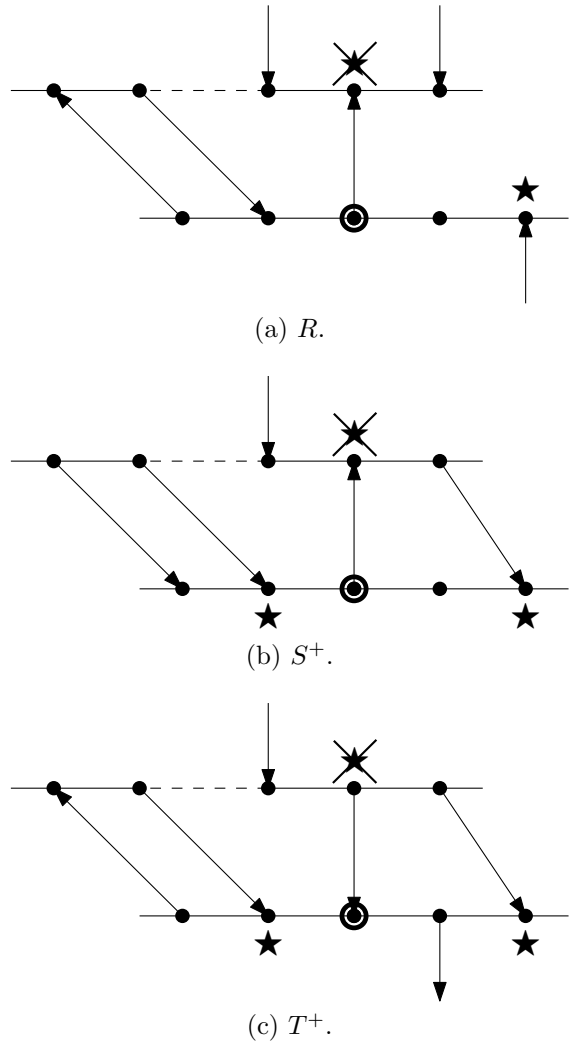

    \centering
    \hf\sfig{57}{$R$.}\hf\\
    \hf\sfig{58}{$S^+$.}\hf\\
    \hf\sfig{59}{$T^+$.}\hf
    \caption{Events used in the proof of
      Proposition~\ref{p:receptive}(ii) for vertices of type 0. Only
      the possibility that $\mate{u}{-2}=(\mate{u}{-})_-$ is shown, but the events
      remain valid if $\mate{u}{-2}=(\mate{u}{-})_+$ (i.e., if the chords of $Z$
      incident with $u_-$ and $u_{-2}$ cross).}
    \label{fig:receptive-ii}
  \end{figure}

  Part (iii) follows by considering the event defined by the diagram in
  Figure~\ref{fig:receptive-iii}. Note that the event is regular and
  its probability is at least $1/2^5 = 8/256$. Furthermore, the event
  is favourable for the vertex $u$. Thus, $u$ is 8-receptive.
  \begin{figure}
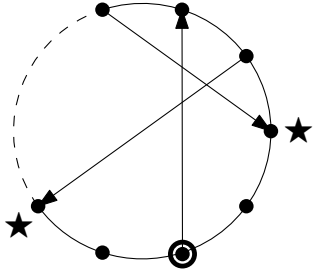

    \centering
    \hf\fig{54}{}\hf
    \caption{The event used in the proof of
      Proposition~\ref{p:receptive}(iii) for vertices of type Ia and
      Ib and their mirror types.}
    \label{fig:receptive-iii}
  \end{figure}
\end{proof}

We now argue that Proposition~\ref{p:receptive} implies the assumption
of Lemma~\ref{l:events} that the receptivity of a deficient vertex
$u_i$ is at least $\eta(u_i)$. By the monotonicity of the ordering
$u_1,\dots,u_r$ and the fact that $\size{\nbrxc{u_i}}\leq 4$ and each
deficient vertex has at least one non-deficient neighbour (namely its
sponsor), we have $\eta(u_i) \leq 3\size{\eps(u_i)}$. From
Proposition~\ref{p:receptive} and the definition of $\eps(u_i)$ (see
the beginning of Section~\ref{sec:nochord} and
Table~\ref{tab:chord-type}), it is easy to check that $u_i$ is
$(3\size{\eps(u_i)})$-receptive, which implies the claim.

Hence, the assumption of Lemma~\ref{l:events} is satisfied. Let
$p(u_i,J_j)$ be the numbers whose existence is guaranteed by
Lemma~\ref{l:events}. We can finally describe Algorithm 2, which
consists of the four phases of Algorithm 1, followed by \textbf{Phase
  5} described below.

Assume a fixed independent set $I=J_j$ was produced by Phase 4 of the
algorithm. We construct a sequence of independent sets
$I^{(0)},\dots,I^{(r)}$. At the $i$-th step of the construction, $u_i$
may or may not be added, and we will ensure that
\begin{equation}
  \label{eq:adding}
  \prob{u_i \text{ is added at $i$-th step}} = p(u_i,J_j).
\end{equation}
At the beginning, we set $I^{(0)} = I$. For $1\leq i \leq r$, we
define $I^{(i)}$ as follows. If $u_i\in I$ or $I$ is not favourable
for $u_i$, we set $I^{(i)} = I^{(i-1)}$. Otherwise,
by~\eqref{eq:adding} and property (iii) of Lemma~\ref{l:events}, the
probability that none of $u_i$'s neighbours has been added before is
at least
\begin{equation*}
  1 - \sum_{u_\ell\in\nbrx{u_i}} p(u_\ell,J_j) \geq p(u_i,J_j).
\end{equation*}
Thus, by including $u_i$ based on a suitably biased independent coin
flip, it is possible to make the probability of inclusion of $u_i$ in
Phase 5 (conditioned on $I=J_j$) exactly equal to $p(u_i,J_j)$. The
output of Algorithm 2 is the set $I' := I^{(r)}$.

We analyze the probability that a deficient vertex $u_i$ is in
$I'$. By Lemma~\ref{l:nochord} and Proposition~\ref{p:32}, 
\begin{equation*}
  \prob{u_i \in I} \geq \frac{88 + \eps(u_i)}{256}.
\end{equation*}
By the above and property (ii) of Lemma~\ref{l:events}, the
probability that $u_i$ is added to $I'$ during Phase 5 equals
\begin{align*}
  \prob{\text{$u_i$ is added in Phase 5}} &= \sum_{j=1}^s
  \cprob{\text{$u_i$ is
      added in Phase 5}}{I=J_j} \cdot \prob{I=J_j}\\
  &= \sum_{j=1}^s p(u_i,J_j)\cdot p(J_j) = \frac{\size{\eps(u_i)}}{256}.
\end{align*}
Since $u_i$ is deficient, $\eps(u_i)<0$; therefore, we obtain
\begin{align*}
  \prob{u_i\in I'} &= \prob{u_i\in I} + \prob{\text{$u_i$ is added in
      Phase 5}} \\
  &\geq \frac{88+\eps(u_i)}{256}-\frac{\eps(u_i)}{256} = \frac{88}{256}.
\end{align*}

If $w$ is a vertex of $G$ which is the sponsor of a (necessarily
unique) deficient vertex $u_i$, then the probability of the removal of
$w$ in Phase 5 is equal to the probability of the addition of $u_i$,
namely $\size{\eps(u_i)}/256$. From Lemma~\ref{l:nochord} and
Proposition~\ref{p:32}, it follows that $\prob{w\in I}$ is high enough
for $\prob{w\in I'}$ to be still greater than or equal to $88/256$.

Finally, if a vertex $w$ is neither deficient nor the sponsor of a
deficient vertex, it is not affected by Phase 5, and hence $\prob{w\in
  I'} \geq 88/256$ as well. Applying Lemma~\ref{l:fraccol} to
Algorithm 2, we infer that $\chi_f(G)\leq 256/88 = 32/11$ as required.


\section{Subcubic graphs}
\label{sec:subcubic}

The generalisation from triangle-free cubic bridgeless graphs to
triangle-free subcubic graphs is perhaps most clear when phrased in
terms of the second equivalent definition of the fractional chromatic
number as given in Lemma~\ref{l:fraccol}.

In Sections~\ref{sec:algorithm}--\ref{sec:phase5}, we showed that for
a bridgeless triangle-free cubic graph $G'$, $\chi_f(G') \leq k :=
32/11$. Therefore, by Lemma~\ref{l:fraccol}, there exists an integer
$N$ such that $kN$ is an integer and we can colour the vertices of
$G'$ using $N$-tuples from $kN$ colours in such a way that adjacent
vertices receive disjoint lists of colours.

We now show that if $G$ is an arbitrary subcubic graph, then
$\chi_f(G)\leq k$. We proceed by induction on the number of vertices
of $G$. The base cases where $\size{V(G)} \leq 3$ are trivial.

Suppose that $G$ has a bridge and choose a block $B_1$ incident with
only one bridge $e$. (Recall that a \emph{block} of $G$ is a maximal
connected subgraph of $G$ without cutvertices.) Let $B_2$ be the other
component of $G-e$. For $i=1,2$, the induction hypothesis implies that
$B_i$ ($i=1,2$) admits a colouring by $N_i$-tuples from a list of
$\lfloor kN_i\rfloor$ colours, for a suitable integer $N_i$. Setting
$N$ to be a common multiple of $N_1$ and $N_2$ such that $kN$ is an
integer, we see that each $B_i$ has an $N$-tuple colouring by colours
$\Setx{1,\dots,kN}$. Furthermore, since $k > 2$, we may permute the
colours used for $B_1$ so as to make the endvertices of $e$ coloured
by disjoint $N$-tuples. The result is a valid $N$-tuple colouring of
$G$ by $kN$ colours, showing $\chi_f(G)\leq k$.

We may thus assume that $G$ is bridgeless; in particular, it has
minimum degree 2 or 3. We may also assume that it contains a vertex of
degree 2 for otherwise we are done by the results of
Sections~\ref{sec:algorithm}--\ref{sec:phase5}. If $G$ contains at
least two vertices of degree 2, we can form a graph $G''$ by taking
two copies of $G$ and joining the two copies of each vertex of degree
2 by an edge. Since $G''$ is a cubic bridgeless supergraph of $G$, we
find $\chi_f(G)\leq k$.

It remains to consider the case where $G$ is bridgeless and contains
exactly one vertex $v_0$ of degree 2. Let $G_0$ be the bridgeless
cubic graph obtained by suppressing $v_0$, and let $e_0$ denote the
edge corresponding to the pair of edges incident with $v_0$ in $G$. By
Theorem~\ref{t:ks}, $G_0$ has a 2-factor $F_0$ containing $e_0$, such
that $E(F_0)$ intersects every inclusionwise minimal edge-cut of size
3 or 4 in $G_0$.

Let $G_1$ be obtained from two copies of $G$ by joining the copies of
$v_0$ by an edge. Thus, $G_1$ is a cubic graph with precisely one
bridge. The 2-factor $F_0$ of $G_0$ yields a 2-factor $F_1$ of $G_1$
in the obvious way. Moreover, it is not hard to see that every
inclusionwise minimal edge-cut of size 3 or 4 in $G_1$ is intersected
by $E(F_1)$. This is all we need to make the argument of
Sections~\ref{sec:algorithm}--\ref{sec:phase5} work even though $G_1$
is not bridgeless. Consequently, $\chi_f(G_1) \leq k$, and since $G$
is a subgraph of $G_1$, we infer that $\chi_f(G)\leq k$ as well. This
finishes the proof of Theorem~\ref{t:main}.


\end{document}